\documentclass[a4paper, 12pt, reqno]{amsart}

\usepackage{amsaddr}
\usepackage{amsmath}
\usepackage{amsthm}
\usepackage{amssymb}
\usepackage{setspace}
\usepackage{enumerate}
\usepackage{comment}
\usepackage{bm}

\newtheorem{thm}{Theorem}[section]
\newtheorem{cor}[thm]{Corollary}
\newtheorem{prop}[thm]{Proposition}
\newtheorem{lem}[thm]{Lemma}

\theoremstyle{definition}

\newtheorem{rem}[thm]{Remark}

\numberwithin{equation}{section}

\newcommand{\pt}{\partial}

\newcommand{\R}{\mathbb{R}}
\newcommand{\C}{\mathbb{C}}

\newcommand{\G}{\mathcal{G}}

\newcommand{\F}{\mathcal{F}}
\newcommand{\Z}{\mathbb{Z}}
\renewcommand{\epsilon}{\varepsilon}

\newcommand{\pa}{\partial}

\DeclareMathOperator{\re}{Re}
\DeclareMathOperator{\im}{Im}

\title[Scattering for fourth-order NLS]
{Scattering of group-invariant solutions below \\ the ground state   for the fourth-order NLS}

\author{Koichi Komada}
\address{Research Institute for Science and Engineering, \\
Waseda University, 3-4-1 Ohkubo, Shinjuku-ku 169-8555, Japan \\
E-mail: k-komada@aoni.waseda.jp}



\date{}

\begin{document}

\maketitle

\vskip5mm
\noindent
{\bf Abstract.} We consider the focusing, $L^2$-supercritical and $\dot{H}^2$-subcritical nonlinear fourth-order Schr\"{o}dinger equation. The scattering of radially symmetric solutions below the ground state was proved by Guo \cite{G} and Dinh \cite{D}. In this paper, we extend the scattering results to group-invariant solutions. In Komada--Masaki \cite{KoMa_N}, the scattering of group-invariant solutions below the ground state was proved under a certain hypothesis. To remove the hypothesis, we establish the non-optimal scattering result for general solutions, where the threshold of action is less than certain fraction of the action of the ground state. This result is analogous to that in Pausader--Shao \cite{PS} for the $L^2$-critical nonlinear fourth-order Schr\"{o}dinger equation.

\section{Introduction}\label{sec:intro}


We consider the focusing nonlinear fourth-order Schr\"{o}dinger equation
\begin{align} \label{eq:4NLS}
      i \pt_t u - \Delta^2 u
      = - |u|^{p-1} u, \ \ \ (t,x) \in \R \times \R^d
\end{align}
with initial data
\begin{align*}
      u(0,\cdot) = u_0(\cdot) \in H^2(\R^d),
\end{align*}
where $u=u(t,x)$ is a complex-valued unknown function and $p$ satisfies
\begin{equation}
      1 + \tfrac{8}{d} < p < \left\{
      \begin{split}
            & \infty && \mbox{if}\ 1 \leq d \leq 4, \\
            & 1 + \tfrac{8}{d-4} && \mbox{if}\ d \geq 5.
      \end{split}
      \right.
\label{eq:intercritical}
\end{equation}
\eqref{eq:4NLS} is locally well-posed in $H^2(\R^d)$ and has three conserved quantities; the mass $M(u)$, the energy $E(u)$ and the momentum $P(u)$:
\begin{align*}
      M(u(t))
      :=&\ \int_{\R^d} |u(t,x)|^2 dx
      = M(u_0), \\
      E(u(t))
      :=&\ \tfrac{1}{2} \int_{\R^d} |\Delta u(t,x)|^2 dx
      - \tfrac{1}{p+1} \int_{\R^d} |u(t,x)|^{p+1} dx
      = E(u_0), \\
      P(u(t))
      :=&\ \im \int_{\R^d} \overline{u}(t,x) \nabla u(t,x) dx
      = P(u_0).
\end{align*}
\eqref{eq:4NLS} is invariant under the scaling transformation
\begin{align*}
      u_{\lambda}(t,x) := \lambda^{\frac{4}{p-1}} u(\lambda^4 t, \lambda x), \ \ \ \lambda > 0.
\end{align*}
By direct computation, 
\begin{align*}
      \|u_{\lambda}(0)\|_{\dot{H}^s(\R^d)}
      = \lambda^{s - \frac{d}{2} + \frac{4}{p-1}} \|u_0\|_{\dot{H}^s(\R^d)}
\end{align*}
for all $s \in \R$. This implies that the critical exponent of homogeneous Sobolev spaces is given by
\begin{align*}
      s_c := \tfrac{d}{2} - \tfrac{4}{p-1}.
\end{align*}
We say that \eqref{eq:4NLS} is $L^2$-supercritical and $\dot{H}^2$-subcritical if $0 < s_c < 2$. This condition holds if and only if $p$ satisfies \eqref{eq:intercritical}.

The classical mathematical model for propagation of intense laser beams in a bulk medium with Kerr nonlinearity is given by the nonlinear Schr\"{o}dinger equation (NLS). To take into account the role of small fourth-order dispersion, Karpman \cite{K}, Karpman--Shagalov \cite{KS} and Fibich--Ilan--Papanicolaou \cite{FIP} studied the fourth-order NLS as follows,
\begin{align}\label{eq:g4NLS}
      i \pa_t u + \mu\Delta u - \epsilon \Delta^2 u = \lambda |u|^{p-1} u.
\end{align}
In Fukumoto--Moffat \cite{FM}, similar equations also appear in the study of a vortex filament in an incompressible fluid.

In this article, we study the global dynamics for \eqref{eq:4NLS} with $p$ satisfying \eqref{eq:intercritical}. 
There are many results on the global dynamics for the fourth-order NLS. For the equation \eqref{eq:g4NLS} in the defocusing case: $\mu\geq0, \epsilon=1, \lambda=1$, Pausader \cite{P07} proved global well-posedness for all radially symmetric initial data in $H^2(\R^d)$ when $d\geq5$ and $1<p\leq1+8/(d-4)$. 
Note that the radial assumption is not necessary for $p<1+8/(d-4)$ and is removed later in \cite{MXZ11, P09JFA} for the remaining $\dot{H}^2$-critical case $p=1+8/(d-4)$ when $d\geq8$. 
Moreover, it was also shown that solutions in those results scatters in the $\dot{H}^2$-critical case. 
We refer the reader to \cite{PX} for the case where $1 \le d \le 4$ and $p > 1 + 8/d$. 
See also \cite{PS} for the high dimensional $L^2$-critical case, i.e., $d \ge 5$ and $p = 1 + 8/d$.

Let us turn to the focusing case: $\mu \ge 0, \epsilon = 1, \lambda = - 1$. 
It is known that ground state solutions play an important role in describing global dynamics of nonlinear dispersive equations. 
This is also the case for the fourth-order nonlinear Schr\"{o}dinger equation. 
Miao--Xu--Zhao \cite{MXZ09} and Pausader \cite{P09DCDS} proved global well-posedness and scattering of radial solutions below the ground state threshold for \eqref{eq:4NLS} in the $\dot{H}^2$-critical case, i.e., $d \ge 5$ and $p = 1 + 8/(d-4)$. 
In \cite{MXZ09,P09DCDS}, the proofs are based on the concentration compactness argument in the spirit of Kenig--Merle \cite{KM}.
In the inter-critical case \eqref{eq:intercritical}, Guo \cite{G} proved the scattering of radial solutions below the ground state when $d\geq2$. The proof in \cite{G} is based on the concentration compactness argument. Dinh \cite{D} provides an alternative proof, avoiding the use of the concentration compactness argument. The proof in \cite{D} relies on the approach introduced by Dodson--Murphy \cite{DM}.

For $\omega>0$, \eqref{eq:4NLS} has standing wave solutions $u(t,x) = e^{i \omega t} \phi(x)$, where $\phi$ is a solution to the fourth-order elliptic equation
\begin{align}\label{eq:SW}
      - \omega \phi - \Delta^2 \phi + |\phi|^{p-1} \phi = 0,\ \ \ x \in \R^d.
\end{align}
For each $\omega > 0$, we define a ground state $Q_{\omega}$ as a solution of \eqref{eq:SW} such that
\begin{align*}
      S_{\omega}(Q_{\omega})
      = \inf\{ S_{\omega}(\phi) \mid 
      \phi \in H^2(\R^d) \setminus \{0\},\ \phi\ \mbox{solves\ \eqref{eq:SW}} \},
\end{align*}
where $S_{\omega}$ is the action functional defined by
\begin{align*}
      S_{\omega}(\phi) :=  \tfrac{\omega}{2} M(\phi) + E(\phi).
\end{align*}
It is known that there exists a ground state solution to \eqref{eq:SW} for all $p > 1$ and $\omega > 0$ (see \cite{BCSN, BN}, for instance). 
Let $K$ be the virial functional defined by
\begin{align*}
      K(\phi)
      :=&\ \tfrac{\pa}{\pa\lambda} E (\lambda^{\frac{d}{2}} \phi(\lambda x))|_{\lambda=1} 
      \notag \\
      =&\ 2 \int_{\R^d} |\Delta \phi(x)|^2 dx
      - \tfrac{d(p-1)}{2(p+1)} \int_{\R^d} |\phi(x)|^{p+1} dx.
\end{align*}
Then we note that every solution $\phi$ of \eqref{eq:SW} satisfies $K(\phi)=0$. Moreover, if $p$ satisfies \eqref{eq:intercritical}, then for all $\omega>0$, 
\begin{align}\label{eq:def_gs_2}
      S_{\omega}(Q_{\omega})
      = \inf \{ S_{\omega}(\phi) \mid \phi \in H^2(\R^d) \setminus \{0\},\ K(\phi) = 0 \}.
\end{align}

In \cite{D,G}, the following result was obtained.

\begin{thm}[\cite{D, G}] \label{thm:radial}
Let $d \ge 2$ and $p$ satisfy \eqref{eq:intercritical}.
If $u_0 \in H^2(\R^d)$ is radial and satisfies
\begin{align} \label{eq:below_gs}
      S_{\omega}(u_0) < S_{\omega}(Q_{\omega}) 
      \ \ \mbox{for some}\ \ \omega > 0,
      \ \ \mbox{and}\ \ 
      K(u_0) \ge 0,
\end{align} 
then there exists a unique global solution $u \in C(\R, H^2(\R^d))$ to \eqref{eq:4NLS} with initial data $u(0)=u_0$. Moreover, $u$ scatters forward and backward in time, i.e., there exist $\phi^{\pm} \in  H^2(\R^d)$ such that
\begin{align}\label{eq:scattering}
      \lim_{t \rightarrow \pm \infty} \|u(t) - e^{-it \Delta^2} \phi^{\pm}\|_{H^2} = 0.
\end{align}
\end{thm}

\begin{rem}
In \cite{D,G}, to be precise, scattering was obtained for radial initial data $u_0$ satisfying 
\begin{equation}\label{eq:scale_prod}
      \left\{
      \begin{split}
            M(u_0)^{\frac{2-s_c}{s_c}} E(u_0) < &\ M(Q)^{\frac{2-s_c}{s_c}} E(Q), \\
            \|u_0\|_{L^2}^{\frac{2-s_c}{s_c}} \|\Delta u_0\|_{L^2}
            < &\ \|Q\|_{L^2}^{\frac{2-s_c}{s_c}} \|\Delta Q\|_{L^2},
      \end{split}
      \right.
\end{equation}
where $Q$ is a ground state solution to \eqref{eq:SW} for $\omega = 1$.
We note that \eqref{eq:scale_prod} is equivalent to the condition \eqref{eq:below_gs}.
\end{rem}

\begin{rem}
The existence of radially symmetric ground state solutions to \eqref{eq:SW} implies that the bound on action in \eqref{eq:below_gs} is sharp. Boulenger--Lenzmann \cite{BL} proved the existence of radial  ground states for \eqref{eq:SW} when $d \ge 1$ and $p$ is odd integer. However, if $p$ is not odd, it remains unknown due to the inapplicability of classical rearrangement techniques caused by the presence of the fourth-order term. In \cite{BL}, the proof relies on the symmetric-decreasing rearrangement in Fourier space, which is limited to the case where the nonlinearity is a polynomial and hence applicable only when $p$ is odd.
\end{rem}

\begin{rem}
Boulenger--Lenzmann \cite{BL} proved that solutions to \eqref{eq:4NLS} with radial initial data $u_0\in H^2(\R^d)$ satisfying 
\begin{align*}
      S_{\omega}(u_0) < S_{\omega}(Q_{\omega})
      \ \ \mbox{for some}\ \ \omega > 0, \ \ \mbox{and}\ \ 
      K(u_0) < 0
\end{align*}
blows up in finite time when $d\geq2$, $p$ satisfies \eqref{eq:intercritical} and $p\leq9$.
\end{rem}


It is worth noting that in both \cite{D,G}, the radially symmetric assumption is retained, and the one-dimensional case is excluded. In this paper, we extend the scattering results for radial solutions to the result for group-symmetric solutions.

Let $M_d(\R)$ be the set of $d\times d$ matrices with real entries.
Let $O(d)$ denote the set of $d\times d$ orthogonal matrices, i.e.,
\begin{align*}
      O(d) := \{\mathcal{R} \in M_d(\R) \mid \mathcal{R}^{T} \mathcal{R} = \mathcal{I}_d \},
\end{align*}
where the transpose of a matrix $\mathcal{A}$ is written by $\mathcal{A}^{T}$ and $\mathcal{I}_d$ denotes the $d\times d$ identity matrix. 
Note that $O(1)=\{\pm1\}$.
We consider a subgroup $G$ of $\R / 2 \pi \Z \times O(d)$, where $\R / 2 \pi \Z \times O(d)$ is a group with the binary operation $(+,\cdot)$. 
Throughout this paper, we assume that for any $(\theta_1,\G_1),(\theta_2,\G_2) \in G$, $\G_1=\G_2$ implies  $\theta_1=\theta_2$. 
This assumption reads as  $\theta$-part is given as a function of $\G$,  and hence
we can use the notation $\G$ without confusion to denote not only a matrix but also an element of $G$. In the sequel, we freely use this identification.
For a subgroup $G$ of $\R / 2 \pi \Z \times O(d)$, we define the Sobolev space with $G$-invariance by
\begin{align*}
      H_G^k(\R^d) := \{ f \in H^k(\R^d) \mid f = \G f,\ \forall \G \in G \},
\end{align*}
where $\G f(x) := e^{-i\theta} (f\circ\G^{-1}) (x) = e^{-i\theta} f (\G^{-1} x)$ for $\G = (\theta,\G) \in G$. 
We remark that the above assumption of $G$ is a natural one since one has $H_G^k(\R^d)=\{0\}$ if it fails. 
We also remark that the radial case corresponds to the case $d\ge2$ and $G_{\mathrm{rad}}=\{0\} \times O(d)$.
Other simple examples of $G$ in $d=1$ are
$G_{\mathrm{even}} = \{ (0,1),(0,-1) \}$
and $G_{\mathrm{odd}} = \{ (0,1),(\pi,-1) \}$,
which give even subspace and odd subspace of Sobolev space $H^k(\R)$, respectively.

Note that since the Cauchy problem of \eqref{eq:4NLS} with $p$ satisfying \eqref{eq:intercritical} is well-posed in $H^2(\R^d)$, it is obvious that if the initial data $u_0$ belongs to $H_G^2(\R^d)$, then the corresponding solution $u$ to \eqref{eq:4NLS} is also $G$-invariant, i.e., $u(t)=\G u(t)$ for all $\G\in G$, due to the fact that the operators $\Delta$ and $\Delta^2$ are invariant for group actions by $\R/2\pi\Z\times O(d)$ and the nonlinear term of \eqref{eq:4NLS} is gauge invariant.

Our main result is stated as follows.

\begin{thm} \label{thm:G-inv}
Let $d \geq 1$ and $p$ satisfy \eqref{eq:intercritical}.
Let $G$ be a subgroup of $\R / 2 \pi \Z \times O(d)$ such that $\inf_{|x|=1} \# G x = \inf_{|x|=1} \# \{ \G  x \mid \G \in G \} \geq 4^{\frac{2}{p-1}}$.
If $u_0\in H^2(\R^d)$ is $G$-invariant and satisfies \eqref{eq:below_gs}, then there exists a unique global solution $u\in C(\R, H^2(\R^d))$ to \eqref{eq:4NLS} with initial data $u(0)=u_0$. Moreover, $u$ scatters forward and backward in time.
\end{thm}

\begin{rem}
The condition $\inf_{|x|=1}\# Gx \ge 4^{\frac{2}{p-1}}$ comes from the use of Theorem \ref{thm:main} in the proof of Theorem \ref{thm:G-inv}. See below for more details on this technical assumption.
\end{rem}

\begin{rem}
When $d \geq 2$, $\inf_{|x|=1} \# G_{\mathrm{rad}} x = \infty$. Hence, Theorem \ref{thm:G-inv} clearly includes the previous Theorem \ref{thm:radial}.
When $d=1$, the condition \eqref{eq:intercritical} leads to $p>9$, and so, $\inf_{|x|=1} \# G_{\mathrm{even}} x = \inf_{|x|=1} \# G_{\mathrm{odd}} x = 2 > 4^{\frac{2}{p-1}}$. Therefore, Theorem \ref{thm:G-inv} proves the scattering of even or odd solutions below the ground state in the one dimensional case, which had not been addressed in previous studies \cite{D, G}.
\end{rem}

\begin{rem}
Inui \cite{Inui17, Inui18} studied the global dynamics of group-invariant solutions for the nonlinear Schr\"{o}dinger equation and proved the scattering of solutions above the ground state. We remark that these results for group-invariant solutions rely on the known results in \cite{AN, DHR, FXC}, where the scattering of general solutions below the ground state was proved.
\end{rem}

In Komada--Masaki \cite{KoMa_N}, the scattering result for group-invariant solutions to \eqref{eq:4NLS} under a certain hypothesis was obtained (see also Komada--Masaki \cite{KoMa_DCDS} for even solutions in $d=1$). To explain the results in \cite{KoMa_N}, we define the thresholds of action for general solutions and group-invariant solutions, respectively. If an initial data $u_0$ satisfies \eqref{eq:below_gs}, then by using the conservation laws and by applying the variational analysis, we see that the corresponding solution $u$ exists globally in time and $\|u\|_{L_t^{\infty}H_x^2(\R\times\R^d)}<\infty$ (see Lemma \ref{lem:invariant_set} for instance).
For any interval $I\subset\R$, we define a scattering size by
\begin{align*}
      \|u\|_{X(I)}
      := \|u\|_{L_t^{q_1} L_x^r(I\times\R^d)},
\end{align*}
where $r=p+1, q_1=\tfrac{4(p-1)(p+1)}{8-(d-4)(p-1)}$.
Then for all global solutions $u$ of \eqref{eq:4NLS}, the boundedness of the scattering size 
\begin{align}\label{eq:scattering_criteria}
      \|u\|_{X(\R)} < \infty
\end{align}
implies the scattering \eqref{eq:scattering} (see Proposition \ref{prop:H^2-scattering}).
For $\omega>0$ and for $L\geq0$, we define
\begin{align*}
      C_{\omega}(L)
      := \sup \{ \|u\|_{X(I)}^{q_1} \mid S_{\omega}(u_0) \leq L,\ K(u_0) \geq 0 \},
\end{align*}
where the supremum is taken over all solutions $u$ to \eqref{eq:4NLS} defined on $I\times\R^d$  with initial data $u_0\in H^2(\R^d)$.
By using Proposition \ref{prop:small_data} and \ref{prop:stability}, we see that $C_{\omega}(L)$ is sublinear for small $L>0$ and that $L\mapsto C_{\omega}(L)$ is non-decreasing and continuous whenever $C_{\omega}(L)$ is finite.
Let
\begin{align}\label{eq:threshold}
      L_{\omega}^*
      := \sup \{ L \in [0,\infty) \mid C_{\omega}(L) < \infty \}.
\end{align}
Then $L_{\omega}^*\geq S_{\omega}(Q_{\omega})$ for all $\omega>0$ implies that for all solutions $u$ to \eqref{eq:4NLS} with initial data $u_0$ satisfying \eqref{eq:below_gs}, \eqref{eq:scattering_criteria} holds and so $u$ scatters. Note that $L_{\omega}^*\leq S_{\omega}(Q_{\omega})$ follows from $C_{\omega}(S_{\omega}(Q_{\omega}))=\infty$, which is a consequence of the fact that $\|e^{i\omega t}Q_{\omega}\|_{X(\R)}=\infty$. For $\omega>0$ and for $L\geq0$, we also define the criterion for $G$-invariant solutions by
\begin{align*}
      C_{G,\omega}(L)
      := \sup \{ \|u\|_{X(I)}^{q_1} \mid S_{\omega}(u_0) \leq L,\ K(u_0) \geq 0 \},
\end{align*}
where the supremum is taken over all $G$-invariant solutions $u$ to \eqref{eq:4NLS} defined on $I\times\R^d$ with initial data $u_0\in H_G^2(\R^d)$.
Let
\begin{align}\label{eq:threshold_G}
      L_{G,\omega}^*
      := \sup \{ L \in [0,\infty) \mid C_{G,\omega}(L) < \infty \}.
\end{align}
Then $L_{G,\omega}^*\geq S_{\omega}(Q_{\omega})$ for all $\omega>0$ implies that for all $G$-invariant solutions $u$ to \eqref{eq:4NLS} with initial data $u_0$ satisfying \eqref{eq:below_gs}, \eqref{eq:scattering_criteria} holds. We note that by the definition, $C_{G,\omega}(L)\leq C_{\omega}(L)$ holds and hence $L_{G,\omega}^*\geq L_{\omega}^*$ for all subgroups $G$ of $\R / 2 \pi \Z \times O(d)$.

For a subgroup $G$ of $\R / 2 \pi \Z \times O(d)$ and for $\omega>0$, we define
\begin{align*}
      m_{G,\omega}
      := \inf_{|x|=1} (\# G x) L_{G_{x},\omega}^*,
\end{align*}
where $\# G x = \# \{ \G x \mid \G \in G \}$ and $G_x = \{ \G \in G \mid \G x = x \}$ is a subgroup of $G$.
In the previous works \cite{KoMa_N, KoMa_DCDS}, the following result for scattering of group-invariant solutions was obtained.

\begin{thm}[\cite{KoMa_N, KoMa_DCDS}]\label{thm:previous}
Let $d\geq1$ and $p$ satisfy \eqref{eq:intercritical}.
Let $G$ be a subgroup of $\R/2\pi\Z\times O(d)$.
For all $\omega>0$, $L_{G,\omega}^*\geq S_{\omega}(Q_{\omega})$ holds true under the hypothesis $L_{G,\omega}^* < m_{G,\omega}$.
\end{thm}

\begin{rem}
If $m_{G,\omega}<\infty$, then $L_{G,\omega}^*\leq m_{G,\omega}$ holds in general (see \cite{KoMa_N}). Hence, the hypothesis $L_{G,\omega}^*<m_{G,\omega}$ in Theorem \ref{thm:previous} reads as the failure of the identity. If $m_{G,\omega}=\infty$, then the hypothesis is not necessary in Theorem \ref{thm:previous} (see also \cite{KoMa_N}).
\end{rem}

\begin{rem}
In the case $L_{G,\omega}^* = m_{G,\omega}$, we have
\begin{align*}
      L_{G,\omega}^* = m_{G,\omega}
      \geq \left( \inf_{|x|=1} \# G x \right) L_{\omega}^*.
\end{align*}
From the small data theory (Proposition \ref{prop:small_data}), we see that there exists $\delta>0$ such that $L_{\omega}^*\geq\delta$. Therefore, if $\inf_{|x|=1}\# Gx$ is sufficiently large depending on $\delta$, we obtain $L_{G,\omega}^*\geq S_{\omega}(Q_{\omega})$.
\end{rem}

In \cite{KoMa_N, KoMa_DCDS}, the proof of Theorem \ref{thm:previous} is based on the concentration compactness argument introduced by Kenig--Merle \cite{KM}, where they proved global well-posedness and scattering of radial solutions below the ground state for the focusing energy-critical nonlinear Schr\"{o}dinger (NLS) equation. We also refer to Holmer--Roudenko \cite{HR} for the focusing 3D cubic NLS in the radial case. Roughly speaking, the argument begins by contradiction and  reformulating the problem as a variational problem. First, we suppose that the threshold for scattering is strictly below the ground state. Then, we can construct a critical non-scattering solution standing exactly at the threshold. If the problem is restricted to the radial case, we can show that the critical solution is spatially localized near the origin uniformly in time. This uniform localization enables us to establish the local virial identity, which leads to a contradiction. In \cite{G}, the proof of Theorem \ref{thm:radial} follows this strategy.

We recall that the scattering results for NLS is extended to the nonradial case (see \cite{AN,D15,D19,DHR,FXC,KV10}). In the nonradial case, the critical solution is localized near some function $x(t)$. For NLS, by using the Galilean invariance, one can control the behavior of $x(t)$, which enables the contradiction argument via the virial identity. However, the fourth-order Schr\"{o}dinger equation \eqref{eq:4NLS} is not Galilean invariant. Due to this lack of the Galilean invariance, the local virial argument does not work for \eqref{eq:4NLS} in the nonradial case.

In the proof of Theorem \ref{thm:previous}, the hypothesis $L_{G,\omega}^*<m_{G,\omega}$ makes it possible to construct a critical solution localized near the origin.
To construct a critical solution, we apply the profile decomposition for $G$-invariant functions to a sequence optimizing $L_{G,\omega}^*$. 
In this decomposition, each profile $\phi^j$ is $G_{w^j}$-invariant for some $w^j \in \R^d$ and appears in the following form:
\begin{align*}
      \sum_{k=1}^{\# G w^j} \G_k^j \left[ \phi^j(x-x_n^j) \right]
      =: \psi_n^j (x),
\end{align*}
where $\G_k^j \in G$ are such that $G=\bigsqcup_{k=1}^{\# G w^j} \G_k^j G_{w^j}$ and $\psi_n^j$ is $G$-invariant.
By applying the concentration compactness argument, we can show that the optimizing sequence $u_{n,0}$ is consist of only one $\psi_n^1$. 
If $|x_n^1| \to \infty$, we can assume that $w^1 \in \mathbb{S}^{d-1} := \{ x \in \R^d \mid |x|=1 \}$ and $L_{G,\omega}^{*} = \lim_{n \to \infty} S_{\omega}(u_{n,0}) = (\# G w^1) S_{\omega}(\phi^1)$.
Then we have $S_{\omega}(\phi^1) < L_{G_{w^1},\omega}^{*}$ under the hypothesis $L_{G,\omega}^{*} < m_{G,\omega}$.
 This implies that the solution to \eqref{eq:4NLS} with initial data $\phi^1$ scatters, which leads to a contradiction.
If $x_n^1$ is bounded, we see that $w^1 = 0$ and $\psi_n^1 = \phi^1$.
In this case, we obtain a critical solution localized near the origin.

For the above reasons, Theorem \ref{thm:previous} proves $L_{G,\omega}^*\geq S_{\omega}(Q_{\omega})$ only under the hypothesis $L_{G,\omega}^*<m_{G,\omega}$.
On the other hand, by using the following theorem, the failure of the hypothesis derives a lower bound for the threshold $L_{G,\omega}^*$ that leads to Theorem \ref{thm:G-inv}.

\begin{thm} \label{thm:main}
Let $d \geq 1$ and $p$ satisfy \eqref{eq:intercritical}.
If $u_0 \in H^2(\R^d)$ satisfies
\begin{align}\label{eq:fraction_gs}
      S_{\omega} (u_0) \leq \left( \tfrac14 \right)^{\frac{2}{p-1}} S_{\omega} (Q_{\omega}) 
      \ \ \mbox{for some}\ \ \omega > 0,
      \ \ \mbox{and}\ \ 
      K(u_0) \geq 0,
\end{align} 
then there exists a unique global solution $u\in C(\R, H^2(\R^d))$ to \eqref{eq:4NLS} with initial data $u(0)=u_0$. Moreover, $u$ scatters forward and backward in time.
\end{thm}

\begin{rem}
Theorem \ref{thm:main} is equivalent to that $L_{\omega}^*$ is strictly larger than $(\frac14)^{\frac{2}{p-1}}S_{\omega}(Q_{\omega})$ for all $\omega>0$.
\end{rem}

Although the bound on action in the condition \eqref{eq:fraction_gs} is not optimal for scattering, we can prove Theorem \ref{thm:G-inv} by using Theorem \ref{thm:previous} and Theorem \ref{thm:main}. \\

\noindent
{\bf Proof of Theorem \ref{thm:G-inv}.}
If $L_{G,\omega}^*<m_{G,\omega}$, then by Theorem \ref{thm:previous}, we have $L_{G,\omega}^*\geq S_{\omega}(Q_{\omega})$.
On the other hand, if $L_{G,\omega}^*=m_{G,\omega}$, by Theorem \ref{thm:main}, we obtain
\begin{align*}
      L_{G,\omega}^*
     = m_{G,\omega}
     \geq \left( \inf_{|x|=1} \# G x \right) L_{\omega}^*
     > \left( \inf_{|x|=1} \# G x \right) \left( \tfrac14 \right)^{\frac{2}{p-1}} S_{\omega}(Q_{\omega}).
\end{align*}
Therefore, if $\inf_{|x|=1} \# G x \geq 4^{\frac{2}{p-1}}$, we have $L_{G,\omega}^*\geq S_{\omega}(Q_{\omega})$ whether $L_{G,\omega}^* < m_{G,\omega}$ holds or not. This completes the proof of Theorem \ref{thm:G-inv}.
\qed \\

Let us give some explanation about the outline of the proof of Theorem \ref{thm:main}. The proof follows the argument analogous to that by Pausader--Shao \cite{PS} for the $L^2$-critical fourth-order Schr\"{o}dinger equation, namely, \eqref{eq:4NLS} with $p=1+8/d$. When $d\geq5$, they proved that if $u_0\in L^2(\R^d)$ satisfies 
\begin{align}\label{eq:mass-critical}
      M(u_0) \leq \left( \tfrac14 \right)^{\frac{d}{4}} M(Q),
\end{align}
where $Q$ is a ground state of \eqref{eq:SW} with $\omega=1$, then the solution $u$ with initial data $u_0$ globally exists and scatters. In the proof, they estimate the mass ratio between $Q$ and the critical non-scattering solution $u$ obtained by the concentration compactness argument. For any real-valued function $\phi(x)$, the sharp Gagliardo--Nirenberg inequality yields that
\begin{align}\label{eq:sharp_GN}
      2 E ( e^{i \phi (x)} u )
      \geq \left\{ \left( \tfrac{M(Q)}{M(u)} \right)^{\frac4d} - 1 \right\} 
      \left( \|\Delta u\|_{L^2}^2 - 2 E(u) \right).
\end{align}
By applying the interaction virial identity to the critical solution $u$ and by choosing a suitable $\phi(x)$, Pausader--Shao obtained the estimate
\begin{align*}
      \left\langle 2 E (e^{i \phi (x)} u) \right\rangle
      \leq 6 \left( \sqrt{2 E(u) \left\langle \|\Delta u\|_{L^2}^2 \right\rangle} - 2 E(u) \right),
\end{align*}
where $\langle F \rangle$ represents the large time average of $F(t)$. Thus, it follows that
\begin{align*}
      \left( \tfrac{M(Q)}{M(u)} \right)^{\frac4d}
      \leq \tfrac{\left\langle 2E(e^{i \phi(x)} u) \right\rangle}
      {\left\langle \|\Delta u\|_{L^2}^2 \right\rangle - 2 E(u)} + 1
      <\tfrac62 + 1 = 4,
\end{align*}
where $\langle \|\Delta u\|_{L^2}^2 \rangle > 2E(u)$ was used in the second inequality.
This implies that the threshold of mass for scattering is larger than the right hand side of \eqref{eq:mass-critical}.

In the proof of Theorem \ref{thm:main}, we use the inequality
\begin{align}\label{eq:var_ineq}
      K ( e^{i \phi(x)} u )
      \geq \left\{ \left( \tfrac{S_{\omega}(Q_{\omega})}{S_{\omega}(u)} \right)^{\frac{p-1}{2}} - 1 \right\}
      \left( 2 \|\Delta u\|_{L^2}^2 - K(u) \right),
\end{align}
which holds if $K(u)>0$, to estimate the ratio of actions between the ground state $Q_{\omega}$ and the critical solution $u$ such that $S_{\omega}(u)=L_{\omega}^*$. We remark that multiplying by $e^{i\phi(x)}$ does not change the mass $M(u)$ but changes the action $S_{\omega}(u)$. Hence, to obtain \eqref{eq:var_ineq}, it is not enough to simply apply variational inequalities to $e^{i\phi(x)}u$, as to obtain \eqref{eq:sharp_GN}. To get \eqref{eq:var_ineq}, we apply the inequality in Lemma \ref{lem:coercivity} to $u$ if 
$\|\Delta u\|_{L^2}^2 \leq \|\Delta(e^{i \phi(x)} u)\|_{L^2}^2$,
and apply it to $e^{i\phi(x)}u$ if otherwise (see the proof of Lemma \ref{lem:kappa} for the details).

The rest of the paper is organized as follows: 
In Section \ref{sec:preli}, we recall the Strichartz type estimates for the linear fourth-order Schr\"{o}dinger equation. We also give a small data theory and a long-time perturbation theory. 
In Section \ref{sec:variational_analysis}, we study the ground state $Q_{\omega}$ and give some lemmas. 
In Section \ref{sec:concentration_compactness}, we construct a critical non-scattering solution by applying the concentration compactness argument. 
In Section \ref{sec:virial}, we establish the localized interaction virial identity.
Finally, in Section \ref{sec:proof}, we complete the proof of Theorem \ref{thm:main}.

\section{Preliminaries}\label{sec:preli}

We consider the linear fourth-order Schr\"{o}dinger equation
\begin{equation}\label{eq:linear4S}
      \left\{
      \begin{split}
            &i\partial_t u-\Delta^2 u=h(t,x), \ \ \ &&(t,x)\in\R\times\R^d, \\
            &u(0,x)=u_0(x), \ \ \ &&x\in\R^d.
      \end{split}
      \right.
\end{equation}
The linear propagator associated with \eqref{eq:linear4S} is given by
\begin{align*}
      e^{-it\Delta^2}f:=\F^{-1}\left[e^{-it|\xi|^4}\hat{f}(\xi)\right],
\end{align*}
where $\F f=\hat{f}$ is the Fourier transform of $f$ and $\F^{-1}f$ is the inverse Fourier transform of $f$. Since $|e^{-it|\xi|^4}|=1$,
\begin{align}\label{eq:L^2}
      \|e^{-it\Delta^2}f\|_{L^2}=\|f\|_{L^2}.
\end{align}
Ben-Artzi--Koch--Saut \cite{BKS} studied the dispersive estimates for Schr\"{o}dinger equations with fourth-order dispersion. Applying their result, we have
\begin{align}\label{eq:L^infty}
      \|e^{-it\Delta^2}f\|_{L^{\infty}}\lesssim |t|^{-\frac{d}{4}}\|f\|_{L^1}
\end{align}
and 
\begin{align}\label{eq:stationary_phase}
      \|\pa^{\alpha}e^{-it\Delta^2}f\|_{L^{\infty}}\lesssim |t|^{-\frac{d}{2}}\|f\|_{L^1},
\end{align}
for all $\alpha\in\Z^d$ such that $|\alpha|=d$. Interpolating \eqref{eq:L^2} and \eqref{eq:L^infty}, for $2\leq r\leq\infty$,
\begin{align}\label{eq:dispersive}
      \|e^{-it\Delta^2}f\|_{L^r}\lesssim |t|^{-\frac{d}{4}(1-\frac{2}{r})}\|f\|_{L^{r^{\prime}}}.
\end{align}

From the above dispersive estimates, we obtain the following two types of Strichartz estimate.

\begin{lem}[Strichartz estimates for biharmonic admissible pairs]\label{lem:Strichartz_B}
Let $d\geq1$. Suppose that pairs $(q_1,r_1)$ and $(q_2,r_2)$ satisfy
\begin{align}\label{eq:B-admissible}
      2\leq q_j,r_j\leq\infty, \ \ \ (q_j,r_j,d)\neq(2, \infty, 4)
      \ \ \ \mbox{and}\ \ \ \tfrac{4}{q_j}+\tfrac{d}{r_j}=\tfrac{d}{2}, \ \ \ \mbox{for}\ j=1,2.
\end{align}
Then for any $I\subset\R$,
\begin{align*}
      \|u\|_{L_t^{q_1}L_x^{r_1}(I\times\R^d)}
      \lesssim \|u_0\|_{L_x^2(\R^d)}+\|h\|_{L_t^{q_2^{\prime}}L_x^{r_2^{\prime}}(I\times\R^d)}.
\end{align*}
\end{lem}

\noindent
{\bf Proof.} We can prove Lemma \ref{lem:Strichartz_B} by using \eqref{eq:dispersive} and the theorem in Keel--Tao \cite{KT}.
\qed \\

\begin{lem}[Strichartz estimates for Schr\"{o}dinger admissible pairs]\label{lem:Strichartz_S}
Let $d\geq1$. Suppose that pairs $(q_1,r_1)$ and $(q_2,r_2)$ satisfy
\begin{align}\label{eq:S-admissible}
      2\leq q_j,r_j\leq\infty, \ \ \ (q_j,r_j,d)\neq(2, \infty, 2)
      \ \ \ \mbox{and}\ \ \ \tfrac{2}{q_j}+\tfrac{d}{r_j}=\tfrac{d}{2}, \ \ \ \mbox{for}\ j=1,2.
\end{align}
Then for any $I\subset\R$,
\begin{align*}
      \|u\|_{L_t^{q_1}L_x^{r_1}(I\times\R^d)}
      \lesssim \||\nabla|^{-\frac{2}{q_1}}u_0\|_{L_x^2(\R^d)}
      +\||\nabla|^{-\frac{2}{q_1}-\frac{2}{q_2}}h\|_{L_t^{q_2^{\prime}}L_x^{r_2^{\prime}}(I\times\R^d)}.
\end{align*}
\end{lem}

\noindent
{\bf Proof.} Lemma \ref{lem:Strichartz_S} follows from \eqref{eq:stationary_phase} and the theorem in Keel--Tao \cite{KT}.
\qed \\

From the above Strichartz estimates, we have the following local well-posedness of \eqref{eq:4NLS} for $H^2$-initial data.

\begin{prop}[Local well-posedness in $H^2$]\label{prop:LWP}
Let $d\geq1$. Suppose 
\begin{equation*}
      \left\{
      \begin{split}
            &\ p\geq2 \ \ \ &&\mbox{if}\ d=1,2, \\
            &\ 1+\tfrac{2}{d}\leq p< 1+\tfrac{8}{(d-4)_{+}} &&\mbox{if}\ d\geq3.
      \end{split}
      \right.
\end{equation*}
Then, for any $u_{0}\in H^2(\R^d)$, there exist $T_{*}, T^{*}\in(0,\infty)$ and a unique solution $u$ to \eqref{eq:4NLS} with initial data $u(0)=u_0$ such that
\begin{align*}
      u\in C((-T_{*}, T^{*}), H^2(\R^d))\cap L_{loc}^q((-T_{*}, T^{*}), W^{2,r}(\R^d))
\end{align*}
for all $(q,r)$ satisfying \eqref{eq:B-admissible}. Furthermore, we can choose $T_{*}$ and $T^{*}$ to be non-increasing functions of $\|u_0\|_{H^2}$. Moreover, $u$ satisfies $M(u(t))=M(u_0)$, $E(u(t))=E(u_0)$ and $P(u(t))=P(u_0)$ for all $t\in(-T_{*}, T^{*})$.
\end{prop}

\noindent
{\bf Proof.} See \cite{D}.
\qed\\

To prove the scattering result, we need the following Kato type inhomogeneous Strichartz estimate.

\begin{lem}[Kato type estimates]\label{lem:Kato}
Let $d\geq1$. Suppose that $q_1, q_2$ and $r$ satisfy
\begin{equation*}
      1< q_1, q_2<\infty, \ \ \ 
      2< r \left\{
            \begin{split}
                  &\leq\infty \ \ \ &&\mbox{if}\ 1\leq d\leq3, \\
                  &<\infty \ \ \ &&\mbox{if}\ d=4, \\
                  &<\tfrac{2d}{d-4} \ \ \ &&\mbox{if}\ d\geq5
            \end{split}
      \right.
\end{equation*}
and
\begin{align}\label{eq:scaling_condi}
      \tfrac{1}{q_1}+\tfrac{1}{q_2}=\tfrac{d}{4}\left(1-\tfrac{2}{r}\right).
\end{align}
Then  for any $I\subset\R$ and $t_0\in I$,
\begin{align}\label{eq:inhom_str}
      \left\|\int_{t_0}^t e^{-i(t-s)\Delta^2}F(s) ds\right\|_{L_t^{q_1}L_x^{r}(I\times\R^d)}
      \lesssim \|F\|_{L_t^{q_2^{\prime}}L_x^{r^{\prime}}(I\times\R^d)}.
\end{align}
\end{lem}

\noindent
{\bf Proof.} By using \eqref{eq:dispersive}, we have
\begin{align*}
      \left\| \int_{t_0}^{t}e^{-i(t-s)(\Delta^2-\mu\Delta)}F(s)ds \right\|_{L_x^r}
      \lesssim& \int_{t_0}^{t}\left\| e^{-i(t-s)\Delta^2}F(s) \right\|_{L_x^r}ds \notag \\
      \lesssim& \int_{t_0}^{t}|t-s|^{-\frac{d}{4}(1-\frac{2}{r})}\|F(s)\|_{L_x^{r^{\prime}}}ds.
\end{align*}
Hence, from the Hardy--Littlewood--Sobolev inequality and \eqref{eq:scaling_condi}, we obtain \eqref{eq:inhom_str}.
\qed \\

Throughout this paper, for $1+8/d<p<1+8/(d-4)_{+}$ we fix
\begin{align*}
      &r=p+1, &&q=\tfrac{8(p+1)}{d(p-1)}, \notag \\
      &q_1=\tfrac{4(p-1)(p+1)}{8-(d-4)(p-1)}, &&q_2=\tfrac{4(p-1)(p+1)}{(dp-4)(p-1)-8}.
\end{align*}
Then $(q,r)$ satisfies $\frac{4}{q}+\frac{d}{r}=\frac{d}{2}$ and $(q_1,r)$ and $(q_2,r)$ satisfy 
\begin{align*}
      \tfrac{4}{q_1}+\tfrac{d}{r}=\tfrac{4}{p-1}=\tfrac{d}{2}-s_c,\ \ \ 
      \tfrac{4}{q_2}+\tfrac{d}{r}=d-\tfrac{4}{p-1}=\tfrac{d}{2}+s_c.
\end{align*}
In particular, we have $\tfrac{1}{q_1}+\tfrac{1}{q_2}=\tfrac{d}{4}(1-\tfrac{2}{r})$. For any time interval $I\subset\R$, we define function space $X(I)$ with norm
\begin{align*}
      \|u\|_{X(I)}:= \|u\|_{L_t^{q_1}L_x^r(I\times\R^d)}.
\end{align*}
We also define
\begin{align*}
      \|F\|_{N(I)}:=\|F\|_{L_t^{q_2^{\prime}}L_x^{r^{\prime}}(I\times\R^d)}.
\end{align*}

\begin{prop}[Small data]\label{prop:small_data}
Let $d\geq1$, and let $p$ satisfy \eqref{eq:intercritical}. There exists small $\delta_{sd}>0$ such that the following holds. Let $I$ be a time interval which contains $t_0$. If $u_0\in H^2(\R)$ satisfies
\begin{align*}
      \|e^{-it\Delta^2}u_0\|_{X(I)}\leq\delta_{sd},
\end{align*}
then there exists a unique solution $u\in X(I)$ to \eqref{eq:4NLS} with initial data $u(t_0)=u_0$. Moreover,
\begin{align*}
      \|u\|_{X(I)}\leq 2\|e^{-it\Delta^2}u_0\|_{X(I)}.
\end{align*}
\end{prop}

\noindent
{\bf Proof.} We can prove Proposition \ref{prop:small_data} by using Lemma \ref{lem:Kato} and applying the standard contraction mapping principle. 
\qed \\

To prove the scattering, we will use the following criterion.

\begin{prop}[$H^2$-scattering]\label{prop:H^2-scattering}
Let $d\geq1$, and let $p$ satisfy \eqref{eq:intercritical}. If $u\in C(\R,H^2(\R^d))$ is a global solution to \eqref{eq:4NLS} and satisfies
\begin{align*}
      \|u\|_{X(\R)}<\infty, \ \ \ \mbox{and}\ \ \ \|u\|_{L_t^{\infty}H_x^2(\R\times\R^d)}<\infty,
\end{align*}
then $u$ scatters forward and backward in time.
\end{prop}

\noindent
{\bf Proof.} See \cite{D} for the details.
\qed\\

The following long-time perturbation theory plays a crucial role in the concentration compactness argument in Section \ref{sec:concentration_compactness}.

\begin{prop}[Long-time perturbation theory]\label{prop:stability}
Let $d\geq1$, and let $p$ satisfy \eqref{eq:intercritical}. For each $A\gg1$, there exists $\delta_0=\delta_0(A)\ll1$ and $C=C(A)\gg1$ such that the following holds. Let $I\subset\R$ be a compact time interval, and let $\tilde{u}\in C(I,H^2(\R^d))$ be an approximate solution of \eqref{eq:4NLS} in the sense that
\begin{align*}
      i\pt_t \tilde{u}-\Delta^2\tilde{u}=-|\tilde{u}|^{p-1}\tilde{u}+e,
\end{align*}
for some $e\in N(I)$. Let $t_0\in I$ and $u_0\in H^2(\R^d)$. If 
\begin{align*}
      \|\tilde{u}\|_{X(I)}\leq A,\ \ \ \|e\|_{N(I)}\leq\delta
\end{align*}
and 
\begin{align*}
      \|e^{-i(t-t_0)\Delta^2}(\tilde{u}(t_0)-u_0)\|_{X(I)}\leq\delta
\end{align*}
for some $\delta\in(0,\delta_0]$, then there exists a solution $u\in C(I,H^2(\R^d))$ to \eqref{eq:4NLS} such that $u(t_0)=u_0$. Moreover, $u$ satisfies
\begin{align*}
      \|u-\tilde{u}\|_{X(I)}\leq C\delta.
\end{align*}
\end{prop}

\noindent
{\bf Proof.} See \cite{KoMa_N}. \qed

\section{Variational analysis}\label{sec:variational_analysis}

In this section, we study the ground state $Q_{\omega}$ and prepare some variational lemmas.

\begin{lem}\label{lem:def_gs_3}
Let $d\geq1$, and $p$ satisfy \eqref{eq:intercritical}. For all $\omega>0$,
\begin{align*}
      S_{\omega}(Q_{\omega})
      =\inf\{\tilde{S}_{\omega}(\phi)\ |\ \phi\in H^2(\R^d)\setminus\{0\}, \ K(\phi)\leq0 \},
\end{align*}
where
\begin{align*}
      \tilde{S}_{\omega}(\phi)
      :=S_{\omega}(\phi)-\tfrac{2}{d(p-1)}K(\phi)
      =\tfrac{s_c}{d}\int_{\R^d}|\Delta\phi|^2dx
      +\tfrac{\omega}{2}\int_{\R^d}|\phi|^2dx.
\end{align*}
\end{lem}

\noindent
{\bf Proof.} By the definition of $\tilde{S}_{\omega}$, we see that if $K(\phi)=0$, then $\tilde{S}_{\omega}(\phi)=S_{\omega}(\phi)$. Hence, we have
\begin{align}\label{eq:S>}
      S_{\omega}(Q_{\omega})
      \geq&\ \inf\{\tilde{S}_{\omega}(\phi)\ |\ \phi\in H^2(\R^d)\setminus\{0\},\ K(\phi)\leq0\}.
\end{align}
Let $\phi$ be such that $K(\phi)<0$. For any $\lambda>0$, we have
\begin{align*}
      K(\lambda\phi)
      =2\lambda^2\int_{\R^d}|\Delta\phi|^2dx
      -\tfrac{d(p-1)}{2(p+1)}\lambda^{p+1}\int_{\R^d}|\phi|^{p+1}dx.
\end{align*}
Since $K(\phi)<0$, there exists $0<\lambda_0<1$ such that $K(\lambda_0\phi)=0$. Thus,
\begin{align*}
      S_{\omega}(Q_{\omega})\leq S_{\omega}(\lambda_0\phi)
      =\tilde{S}_{\omega}(\lambda_0\phi)=\lambda_0^2\tilde{S}_{\omega}(\phi)
      <\tilde{S}_{\omega}(\phi).
\end{align*}
Therefore, we obtain
\begin{align}\label{eq:S<}
      S_{\omega}(Q_{\omega})
      \leq\inf\{\tilde{S}_{\omega}(\phi)\ |\ \phi\in H^2(\R^d)\setminus\{0\},\ K(\phi)\leq0\}.
\end{align}
Lemma \ref{lem:def_gs_3} follows from \eqref{eq:S>} and \eqref{eq:S<}.
\qed \\

From \eqref{eq:def_gs_2} and conservation laws, for each $\omega>0$, the set of functions satisfying \eqref{eq:below_gs} is invariant under the flow of \eqref{eq:4NLS}.

\begin{lem}\label{lem:invariant_set}
Let $d\geq1$, and $p$ satisfy \eqref{eq:intercritical}. Let $\omega>0$ and let $u\in C(I,H^2(\R^d))$ be a solution to \eqref{eq:4NLS} on a maximal interval $I$ of existence. If $S_{\omega}(u(t_0))<S_{\omega}(Q_{\omega})$ and $K(u(t_0))\geq0$ for some $t_0\in I$, then $S_{\omega}(u(t))<S_{\omega}(Q_{\omega})$ and $K(u(t))\geq0$ for all $t\in I$. In particular, $u$ can be exist globally in time, i.e. $I=\R$.
\end{lem}

\noindent
{\bf Proof.} Suppose that $S_{\omega}(u(t_0))<S_{\omega}(Q_{\omega})$ and $K(u(t_0))\geq0$ for some $t_0\in I$. Then, by the conservation laws, we have $S_{\omega}(u(t))=S_{\omega}(u(t_0))<S_{\omega}(Q_{\omega})$ for all $t\in I$. If $K(u(t_1))<0$ for some $t_1\in I$, then by the continuity of $K(u(t))$, there exists $t_2\in I$ such that $K(u(t_2))=0$. From \eqref{eq:def_gs_2} this implies that $S_{\omega}(u(t_2))\geq S_{\omega}(Q_{\omega})$ or $u(t_2)=0$. $S_{\omega}(u(t_2))\geq S_{\omega}(Q_{\omega})$ is a contradiction and $u(t_2)=0$ implies that $u(t)=0$ and so $K(u(t))=0$ for all $t\in I$, which is also a contradiction.

Now we prove $I=\R$. By using $K(u(t))\geq0$, we have
\begin{align*}
      \|u(t)\|_{H^2(\R^d)}
      \lesssim \tilde{S}_{\omega}(u(t))
      = S_{\omega}(u)-\tfrac{2}{d(p-1)}K(u(t))
      \leq S_{\omega}(u)
\end{align*}
for all $t\in I$. This implies that $\|u(t)\|_{H^2(\R^d)}$ is uniformly bounded in $t\in I$, which proves the global existence.
\qed \\

To prove Theorem \ref{thm:main}, we use the following lemma.

\begin{lem}\label{lem:coercivity}
Let $d\geq1$, and $p$ satisfy \eqref{eq:intercritical}. 
If $K(u)>0$, then for any $\omega>0$,
\begin{align*}
      K(u)> 2\left\{1-\left(\tfrac{S_{\omega}(u)}{S_{\omega}(Q_{\omega})}\right)^{\frac{p-1}{2}}\right\}
      \|\Delta u\|_{L^2}^2.
\end{align*}
\end{lem}

\noindent
{\bf Proof.} 
By using $K(u)>0$, we have
\begin{align*}
      \lambda
      :=\left(\tfrac{S_{\omega}(Q_{\omega})}{S_{\omega}(u)}\right)^{\frac12}
      =\left(\tfrac{S_{\omega}(Q_{\omega})}
      {\tilde{S}_{\omega}(u)+\frac{2}{d(p-1)}K(u)}\right)^{\frac12}
      <\left(\tfrac{S_{\omega}(Q_{\omega})}{\tilde{S}_{\omega}(u)}\right)^{\frac12}.
\end{align*}
Hence,
\begin{align*}
      \tilde{S}_{\omega}(\lambda u)
      = \lambda^2\tilde{S}_{\omega}(u)
      <S_{\omega}(Q_{\omega}).
\end{align*}
From Lemma \ref{lem:def_gs_3}, this implies that $K(\lambda u)>0$. Thus, we obtain
\begin{align*}
      K(u)=&\ \lambda^{-(p+1)}K(\lambda u)
      +2(1-\lambda^{-(p-1)})\|\Delta u\|_{L^2}^2 \\
      >&\ 2(1-\lambda^{-(p-1)})\|\Delta u\|_{L^2}^2.
\end{align*}
This completes the proof of Lemma \ref{lem:coercivity}.
\qed

\section{Concentration compactness}\label{sec:concentration_compactness}

In this section, we construct a critical non-scattering solution with action equal to $L_{\omega}^*$ under the hypothesis $L_{\omega}^*<S_{\omega}(Q_{\omega})$.

\begin{prop}[Profile decomposition]\label{prop:profile_decomp}
Let $d\geq1$ and $p$ satisfy \eqref{eq:intercritical}. Let $\{f_n\}$ be a sequence of functions uniformly bounded in $H^2(\R^d)$. Passing to a subsequence if necessary, there exists $J^{*}\in\{0,1,...\}\cup\{\infty\}$ and for each $j\in[1,J^{*}]$, there exists a function $\phi^j\in H^2(\R^d)\setminus\{0\}$ and a sequence $(t_n^j,x_n^j)\in\R\times\R^d$ such that for each  $1\leq J\leq J^{*}$, we have the decomposition
\begin{align}\label{eq:decomp}
      f_n(x)=\sum_{j=1}^J e^{it_n^j\Delta^2}\phi^j(x-x_n^j)+R_n^J(x)
\end{align}
with the following properties: 
\begin{enumerate}\renewcommand{\labelenumi}{(\roman{enumi})}
\item
\begin{align}\label{eq:smallness}
      \lim_{J\rightarrow J^{*}}\limsup_{n\rightarrow\infty}\|e^{-it\Delta^2}R_n^J\|_{X(\R)}=0.
\end{align}
\item
For each $1\leq J\leq J^{*}$,
\begin{align}\label{eq:norm_decomp}
      &\lim_{n\rightarrow\infty}\left\{
      \|f_n\|_{\dot{H}^s}^2-\sum_{j=1}^J \|\phi^j\|_{\dot{H}^s}^2-\|R_n^J\|_{\dot{H}^s}^2
      \right\}=0
      \ \ \mbox{for all}\ 0\leq s\leq2,\ \ \mbox{and} \\
      &\lim_{n\to\infty}\left\{
      E(f_n)-\sum_{j=1}^J E(e^{it_n^j\Delta^2}\phi^j) -E(R_n^J)
      \right\}=0.
\end{align}
\item
For each $j\in[1,J^{*}]$, either $x_n^j\equiv0$ or $\lim_{n\rightarrow\infty}|x_n^j|=\infty$ holds. Similarly, for each  $j\in[1,J^{*}]$, either $t_n^j\equiv0$, $\lim_{n\rightarrow\infty}t_n^j=-\infty$ or $\lim_{n\rightarrow\infty}t_n^j=+\infty$ holds. Furthermore, if $1\leq j\neq k\leq J^{*}$, then
\begin{align}\label{eq:orthogonality}
      |x_n^j-x_n^k|+|t_n^j-t_n^k|\rightarrow\infty \ \ \ \mbox{as}\ \ \ n\rightarrow\infty.
\end{align}
\end{enumerate}
\end{prop}

\noindent
{\bf Proof.} The proof is standard, and so, we omit the detail. In \cite{KoMa_N}, the profile decomposition for a sequence of $G$-invariant functions for any subgroup $G$ of $\R/2\pi\Z\times O(d)$ have been proved. Therefore, taking $G=\{(0,\mathcal{I}_d)\}$, we obtain Proposition \ref{prop:profile_decomp}.
\qed \\

\begin{lem}[Concentration compactness]\label{lem:concentration_compactness}
Let $d\geq1$, and $p$ satisfy \eqref{eq:intercritical}. Suppose $L_{\omega}^{*}<S_{\omega}(Q_{\omega})$. Let $\{u_n\}\subset C(\R,H^2(\R^d))$ be a sequence of global solutions to \eqref{eq:4NLS} such that 
\begin{align}
      S_{\omega}(u_n)\leq L_{\omega}^{*}\ \ \mbox{and}\ \ 
      K(u_n(0))\geq0
\end{align}
for all $n$. Suppose that $S_{\omega}(u_n)\rightarrow L_{\omega}^{*}$ and there exists $\{t_n\}\subset\R$ such that
\begin{align}
      \lim_{n\rightarrow\infty}\|u_n\|_{X((-\infty,t_n])}
      =\lim_{n\rightarrow\infty}\|u_n\|_{X([t_n,+\infty))}
      =\infty.
\end{align}
Then, there exists $\phi\in H^2(\R^d)$ and $\{x_n\}\subset\R^d$ such that $u_n(t_n,\cdot+x_n)$ has a subsequence that converges to $\phi$ in $H^2(\R^d)$. 
\end{lem}

\noindent
{\bf Proof.} Since $L_{\omega}^{*}<S_{\omega}(Q_{\omega})$, from Lemma \ref{lem:invariant_set} we have $K(u_n(t))\geq0$ for all $t\in\R$. Thus, by using time translation symmetry, we may assume $t_n\equiv0$ without loss of generality. Then
\begin{align}\label{eq:blowup_X}
      \lim_{n\rightarrow\infty}\|u_n\|_{X((-\infty,0])}
      =\lim_{n\rightarrow\infty}\|u_n\|_{X([0,+\infty))}
      =\infty.
\end{align}
Applying Proposition \ref{prop:profile_decomp} to $u_n(0)$, passing to a subsequence if necessary, for each $1\leq J\leq J^{*}$ we decompose
\begin{align}
      u_n(0)=\sum_{j=1}^J e^{it_n^j\Delta^2}\phi^j(x-x_n^j)+R_n^J
\end{align}
with
\begin{align}
      &M(u_n)
      =\sum_{j=1}^J M(\phi^j)+M(R_n^J)+o(1), \label{eq:mass_decomp} \\
      &E(u_n)
      =\sum_{j=1}^J E(e^{it_n^j\Delta^2}\phi^j)
      +E(R_n^J)+o(1), \label{eq:energy_decomp} \\
      &\|u_n(0)\|_{\dot{H}^2}^2
      =\sum_{j=1}^J \|\phi^j\|_{\dot{H}^2}^2+\|R_n^J\|_{\dot{H}^2}^2+o(1),
      \label{eq:kenergy_decomp}
\end{align}
where $o(1)\rightarrow0$ as $n\rightarrow\infty$ for each $1\leq J\leq J^{*}$. From \eqref{eq:mass_decomp} and \eqref{eq:energy_decomp}, we have
\begin{align}
      \lim_{n\rightarrow\infty}\sum_{j=1}^J S_{\omega}(e^{it_n^j\Delta^2}\phi^j)
      +\lim_{n\rightarrow\infty}S_{\omega}(R_n^J)
      = \lim_{n\rightarrow\infty}S_{\omega}(u_n)
      = L_{\omega}^{*}<S_{\omega}(Q_{\omega}). \label{eq:sum_S} 
\end{align}
Since $K(u_n(0))\geq0$ and $S_{\omega}(Q_{\omega})=\tilde{S}_{\omega}(Q_{\omega})$, we see that
\begin{align*}
      \lim_{n\rightarrow\infty}\tilde{S}_{\omega}(u_n(0))
      \leq \lim_{n\rightarrow\infty}S_{\omega}(u_n(0))
      = L_{\omega}^{*}
      < S_{\omega}(Q_{\omega})
      = \tilde{S}_{\omega}(Q_{\omega}).
\end{align*} 
Hence, from \eqref{eq:kenergy_decomp} we have
\begin{align}
      \lim_{n\rightarrow\infty}\sum_{j=1}^J \tilde{S}_{\omega}(\phi^j)
      +\lim_{n\rightarrow\infty}\tilde{S}_{\omega}(R_n^J)
      = \lim_{n\rightarrow\infty}\tilde{S}_{\omega}(u_n(0)) 
      < \tilde{S}_{\omega}(Q_{\omega}). \label{eq:sum_tilde_S}
\end{align}
From \eqref{eq:sum_tilde_S}, for $n$ sufficiently large, $\tilde{S}_{\omega}(e^{it_n^j\Delta^2}\phi^j)=\tilde{S}_{\omega}(\phi^j)<\tilde{S}_{\omega}(Q_{\omega})$ and so $K(e^{it_n^j\Delta^2}\phi^j)>0$ for each $j\in[1,J^{*}]$. Thus, we obtain
\begin{align}
      \limsup_{n\rightarrow\infty}S_{\omega}(e^{it_n^j\Delta^2}\phi^j)
      \geq \tilde{S}_{\omega}(\phi^j)>0
      \ \ \ \mbox{for all}\ j\in[1,J^{*}]. \label{eq:S_j>0}
\end{align}
Similarly, we have 
\begin{align*}
      \limsup_{n\to\infty}S_{\omega}(R_n^J)\geq0
      \ \ \ \mbox{for all}\ J\in[1,J^{*}].
\end{align*}

Now we consider the following three cases: $J^{*}=0$, $J^{*}=1$ and $J^{*}\geq2$. We will show that the first case and the third case lead to contradiction. In the second case, we will get desired $\phi\in H^2(\R^d)$.

{\bf Case 1:} $J^{*}=0$. In this case, we have
\begin{align*}
      \limsup_{n\rightarrow\infty}\|e^{-it\Delta^2}u_n(0)\|_{X(\R)}
      =\limsup_{n\rightarrow\infty}\|e^{-it\Delta^2}R_n^0\|_{X(\R)}
      =0.
\end{align*}
Hence, for $n$ sufficiently large, we obtain
\begin{align*}
      \|e^{-it\Delta^2}u_n(0)\|_{X(\R)}\leq \delta_{sd}.
\end{align*}
Thus, by using Proposition \ref{prop:small_data}, we have
\begin{align*}
      \|u_n\|_{X(\R)}\leq2\|e^{-it\Delta^2}u_n(0)\|_{X(\R)}\leq2\delta_{sd}.
\end{align*}
This contradicts \eqref{eq:blowup_X} and so $J^{*}=0$ does not occur.

{\bf Case 2:} $J^{*}=1$. In this case, there exists a single profile in the decomposition and so we can write
\begin{align*}
      u_n(0)=e^{it_n\Delta^2}\phi(x-x_n)+R_n.
\end{align*}
Passing to a subsequence, we may assume that either $t_n\equiv0$, $\lim_{n\rightarrow\infty}t_n=-\infty$ or $\lim_{n\rightarrow\infty}t_n=+\infty$ holds. If $t_n\rightarrow-\infty$, then by the Strichartz estimate and the monotone convergence theorem,
\begin{align*}
      \|e^{-i(t-t_n)\Delta^2}\phi(\cdot-x_n)\|_{X([0,+\infty))}
      = \|e^{-it\Delta^2}\phi(\cdot-x_n)\|_{X([-t_n,+\infty))}
      \rightarrow 0
\end{align*}
as $n\rightarrow\infty$. From this and \eqref{eq:smallness}, we have $\limsup_{n\rightarrow\infty}\|e^{-it\Delta^2}u_n(0)\|_{X([0,+\infty))}=0$. Hence, by Proposition \ref{prop:small_data}, for $n$ sufficiently large, we obtain 
\begin{align*}
      \|u_n\|_{X([0,+\infty))}\leq2\delta_{sd}.
\end{align*}
This contradicts \eqref{eq:blowup_X} and so $t_n\rightarrow-\infty$ does not occur. By the similar argument, $t_n\rightarrow+\infty$ does not occur. Therefore, we may assume that $t_n\equiv0$. Thus, we have
\begin{align*}
      u_n(0)=\phi(\cdot-x_n)+R_n.
\end{align*}

Now, we claim $S_{\omega}(\phi)=L_{\omega}^{*}$. To prove this, suppose $S_{\omega}(\phi)<L_{\omega}^{*}$ and let $v$ be a solution to \eqref{eq:4NLS} with $v(0)=\phi$. Then by the definition of $L_{\omega}^{*}$, we have $\|v\|_{X(\R)}<\infty$. Let
\begin{align*}
      \tilde{u}_n(t,x):=v(t,x-x_n)+e^{-it\Delta^2}R_n(x)
\end{align*}
and 
\begin{align*}
      e_n:=&\ i\pt_t \tilde{u}_n-\Delta^2 \tilde{u}_n
      +|\tilde{u}_n|^{p-1}\tilde{u}_n 
      = |\tilde{u}_n|^{p-1}\tilde{u}_n-|v|^{p-1}v.
\end{align*}
Then $\tilde{u}_n(0)=u_n(0)$ and by Lemma \ref{lem:Kato} and the H\"{o}lder inequality, we have
\begin{align*}
      \|e_n\|_{N(\R)}
      \lesssim&\ \left(\|v\|_{X(\R)}+\|e^{-it\Delta^2}R_n\|_{X(\R)}\right)^{p-1}
      \|e^{-it\Delta^2}R_n\|_{X(\R)}.
\end{align*}
Hence, using $\|v\|_{X(\R)}<\infty$ and $\limsup_{n\rightarrow\infty}\|e^{-it\Delta^2}R_n\|_{X(\R)}=0$, we obtain 
\begin{align*}
      \limsup_{n\rightarrow\infty}\|e_n\|_{N(\R)}=0.
\end{align*}
Thus, by using Proposition \ref{prop:stability}, 
\begin{align}\label{eq:X_tilde_u_n}
      \|u_n\|_{X(\R)}
      \leq \|\tilde{u}_n\|_{X(\R)}+o(1)
\end{align}
as $n\rightarrow\infty$. Since $\|\tilde{u}_n\|_{X(\R)}\leq \|v\|_{X(\R)}+\|e^{-it\Delta^2}R_n\|_{X(\R)}<\infty$ for sufficiently large $n$, \eqref{eq:X_tilde_u_n} contradicts \eqref{eq:blowup_X} and so $S_{\omega}(\phi)=L_{\omega}^{*}$.

Therefore, we have $\limsup_{n\rightarrow\infty}\|R_n\|_{H^2}=0$ and so
\begin{align*}
      u_n(0)=\phi(\cdot-x_n)+o(1)
\end{align*}
in $H^2(\R^d)$. The proof of Lemma \ref{lem:concentration_compactness} is complete if we have proved that $J^{*}\geq2$ does not occur.

{\bf Case 3:} $J^{*}\geq2$. In this case, there exist more than one profiles, and from \eqref{eq:sum_S} and \eqref{eq:S_j>0} there exists $\epsilon>0$ such that 
\begin{align}\label{eq:S_of_v^j}
      0<\limsup_{n\rightarrow\infty}S_{\omega}(e^{it_n^j\Delta^2}\phi^j)
      \leq L_{\omega}^{*}-\epsilon
\end{align}
for all $1\leq j\leq J^{*}$. For each $1\leq j\leq J^{*}$, passing to a subsequence, we may assume that $t_n^j$ has a limit $t_0^j\in\{-\infty,0,\infty\}$. By using Lemma \ref{prop:small_data}, we can find an unique solution $v^j$ of \eqref{eq:4NLS} defined on a neighborhood of $-t_0^j$ such that
\begin{align}\label{eq:H^2_v^j}
      \|v^j(-t_n^j)-e^{it_n^j\Delta^2}\phi^j\|_{H^2}\rightarrow0\ \ \ \mbox{as}\ n\rightarrow\infty.
\end{align}
From \eqref{eq:S_of_v^j} and \eqref{eq:H^2_v^j} we have
\begin{align*}
      S_{\omega}(v^j)\leq L_{\omega}^{*}-\epsilon<S_{\omega}(Q_{\omega}).
\end{align*}
Since $K(e^{it_n^j\Delta^2}\phi^j)>0$ for $n$ sufficiently large, from \eqref{eq:H^2_v^j} we have
\begin{align*}
      K(v^j(-t_n^j))>0
\end{align*}
for $n$ sufficiently large. Hence, from Lemma \ref{lem:invariant_set}, we see that $v^j$ exists globally in time. Let $v_n^j:=v^j(t-t_n^j,x-x_n^j)$, and for each $1\leq J\leq J^{*}$ we define
\begin{align*}
      \tilde{u}_n^J:=\sum_{j=1}^J v_n^j+e^{-it\Delta^2}R_n^J.
\end{align*}
Then 
\begin{align*}
      \|\tilde{u}_n^J\|_{X(\R)}
      \leq&\ \left\| \sum_{j=1}^J v_n^j \right\|_{X(\R)}+\|e^{-it\Delta^2}R_n^J\|_{X(\R)}.
\end{align*}
Since $C_{\omega}(L)$ is sublinear around $0$ and bounded on $[0,L_{\omega}^{*}-\epsilon]$, from \eqref{eq:S_of_v^j} and \eqref{eq:H^2_v^j} we obtain
\begin{align*}
      \sum_{j=1}^J \|v^j\|_{X(\R)}^{q_1}
      \leq&\ \sum_{j=1}^J C_{\omega}\left(S_{\omega}(v^j)\right) 
      \lesssim_{L_{\omega}^{*},\epsilon}\ \sum_{j=1}^J S_{\omega}(v^j) 
      \lesssim_{L_{\omega}^{*},\epsilon}\ 1.
\end{align*}
Hence, by using \eqref{eq:orthogonality}, we have
\begin{align} \label{eq:X_sum_v_n^j}
      \limsup_{n \to \infty} \left\| \sum_{j=1}^{J} v_n^j \right\|_{X(\R)}
      = \left( \sum_{j=1}^{J} \|v^j\|_{X(\R)}^{q_1} \right)^{\frac{1}{q_1}}
      \lesssim_{L_{\omega}^{*}, \epsilon} 1
\end{align}
for all $1 \le J \le J^{*}$.
Therefore, we obtain
\begin{align}\label{eq:X_tilde_u_n^J}
      \lim_{J\rightarrow J^{*}}\limsup_{n\rightarrow\infty}\|\tilde{u}_n^J\|_{X(\R)}
      \lesssim_{L_{\omega}^{*},\epsilon}1.
\end{align}
We define
\begin{align*}
      e_n^J
      :=&\ i\pt_t \tilde{u}_n^J-\Delta^2 \tilde{u}_n^J
      +|\tilde{u}_n^J|^{p-1}\tilde{u}_n^J \\
      =&\ F\left(\sum_{j=1}^J v_n^j+e^{-it\Delta^2}R_n^J\right)-\sum_{j=1}^J F(v_n^j), 
\end{align*}
where $F(z)=|z|^{p-1}z$ for $z\in\C$. Since
\begin{align*}
      &\left\|F\left(\sum_{j=1}^J v_n^j +e^{-it\Delta^2}R_n^J\right)
      -F\left(\sum_{j=1}^J v_n^j\right)\right\|_{N(\R)} \notag \\
      &\lesssim\ \left\| \left( \left|\sum_{j=1}^J v_n^j\right|^{p-1}
      +|e^{-it\Delta^2}R_n^J|^{p-1} \right)
      |e^{-it\Delta^2}R_n^J| \right\|_{N(\R)} \notag \\
      &\lesssim\ \left( \left\| \sum_{j=1}^J v_n^j \right\|_{X(\R)}^{p-1}
      +\|e^{-it\Delta^2}R_n^J\|_{X(\R)}^{p-1} \right)
      \|e^{-it\Delta^2}R_n^J\|_{X(\R)},
\end{align*}
from \eqref{eq:smallness} and \eqref{eq:X_sum_v_n^j}, we have
\begin{align*}
      \lim_{J\rightarrow J^{*}}\limsup_{n\rightarrow\infty}
      \left\|F\left(\sum_{j=1}^J v_n^j +e^{-it\Delta^2}R_n^J\right)
      -F\left(\sum_{j=1}^J v_n^j\right)\right\|_{N(\R)}
      =0.
\end{align*}
By using \eqref{eq:orthogonality}, we obtain
\begin{align*}
      \lim_{J\to J^*}\limsup_{n\to\infty}
      \left\| F\left(\sum_{j=1}^J v_n^j\right)-\sum_{j=1}^J F(v_n^j) \right\|_{N(\R)} 
      \lesssim&\ \lim_{J\to J^*}\limsup_{n\to\infty}
      \sum_{1\leq j\neq k\leq J} \||v_n^j||v_n^k|^{p-1}\|_{N(\R)} \notag \\
      =&\ 0.
\end{align*}
Therefore, we get
\begin{align}\label{eq:N_e_n^J}
      \lim_{J\rightarrow J^{*}}\limsup_{n\rightarrow\infty}\|e_n^J\|_{N(\R)}
      =0.
\end{align}
From \eqref{eq:H^2_v^j} we have
\begin{align}\label{eq:tilde_u_n^J(0)}
      \tilde{u}_n^J(0)
      =\sum_{j=1}^J v^j(-t_n^j,x-x_n^j) +R_n^J
      =u_n(0)+o_J(1),
\end{align}
where $o_J(1)\to0$ in $H^2(\R^d)$ as $n\to\infty$, for each $1\leq J\leq J^*$. 
From \eqref{eq:X_tilde_u_n^J}, \eqref{eq:N_e_n^J} and \eqref{eq:tilde_u_n^J(0)}, using Proposition \ref{prop:stability}, we have
\begin{align*}
      \limsup_{n\rightarrow\infty}\|u_n\|_{X(\R)}
      \leq \lim_{J\rightarrow J^{*}}\limsup_{n\rightarrow\infty}\|\tilde{u}_n^J\|_{X(\R)}
      \lesssim_{L_{G,\omega}^{*},\epsilon} 1.
\end{align*}
This contradicts \eqref{eq:blowup_X} and so $J^{*}\geq2$ does not occur.
\qed \\

\begin{prop}[Existence and precompactness of critical solution]\label{prop:critical_sol}
Let $d\geq1$, and $p$ satisfy \eqref{eq:intercritical}. Suppose $L_{\omega}^{*}<S_{\omega}(Q_{\omega})$. Then there exists $\phi\in H^2(\R^d)$ such that 
\begin{align}\label{eq:S_of_phi}
      S_{\omega}(\phi)=L_{\omega}^{*}\ \ \mbox{and}\ \ 
      K(\phi)\geq0.
\end{align}
Let $u$ be a solution to \eqref{eq:4NLS} with initial data $\phi$. Then
\begin{align}\label{eq:X_of_u}
      \|u\|_{X((-\infty,0])}=\|u\|_{X([0,+\infty))}=\infty.
\end{align}
Moreover, there exists $x:\R\to\R^d$ such that $\{ u(t,\cdot+x(t))\ |\ t\in\R \}$ is precompact in $H^2(\R^d)$.
\end{prop}

\noindent
{\bf Proof.} Since $C_{\omega}(L):[0,L_{\omega})\to[0,\infty)$ is continuous, non-decreasing and $C_{\omega}(L_{\omega})=\infty$, we can take $L_n$ such that $C_{\omega}(L_n)=n$ and $L_n\nearrow L_{\omega}^*$ as $n\to\infty$.
Then by definition of $C_{\omega}(L_n)$, there exists a sequence of solutions $\{u_n\}\subset H^2(\R^d)$ to \eqref{eq:4NLS} such that $L_{n-1}\leq S_{\omega}(u_n)< L_n$, $K(u_n(0))\geq0$ and $\|u_n\|_{X(\R)}\geq n-\frac12$. 
Let $f(t):=\|u_n\|_{X((-\infty,t])}$. Then we have $f(-\infty)=0$ and $f(+\infty)\geq n-\frac12$.
Since $f(t)$ is continuous, we find $t_n\in\R$ such that 
\begin{align*}
      \tfrac14(n-\tfrac12)\leq f(t_n)\leq \tfrac12(n-\tfrac12),
\end{align*}
and so
\begin{align*}
      \lim_{n\to\infty}\|u_n\|_{X((-\infty,t_n])}
      =\lim_{n\to\infty}\|u_n\|_{X([t_n,+\infty))}
      =\infty.
\end{align*}
Moreover, by Lemma \ref{lem:invariant_set}, $K(u_n(t_n))\geq0$.
Hence, $\{u_n\}$ and $\{t_n\}$ satisfy the hypothesis of Lemma \ref{lem:concentration_compactness}.
Thus, there exists $\phi\in H^2(\R^d)$ and $\{x_n\} \subset \R^d$ such that $u_n(t_n,\cdot+x_n)\to\phi$ in $H^2(\R^d)$.
From this strong convergence, we have \eqref{eq:S_of_phi}.
Let $u$ be a solution to \eqref{eq:4NLS} with $u(0)=\phi$.
Then by Proposition \ref{prop:stability}, we obtain \eqref{eq:X_of_u}.

Now we prove that there exists $x(t)$ such that $\{u(t,\cdot+x(t))|t\in\R\}$ is precompact in $H^2(\R^d)$.
For $u\in H^2(\R^d)$, we define
\begin{align*}
      \lambda(u,R)
      :=&\ \sup_{y\in\R^d}\int_{|x-y|< R}|u|^2+|\Delta u|^2 dx, \\
      \rho(u,\delta)
      :=&\ \inf\{R>0\ |\ \lambda(u,R)>(1-\delta)\|u\|_{H^2}^2\}.
\end{align*}
Note that the supremum in the definition of $\lambda$ is attained since $y\mapsto \int_{|x-y|<R}|u|^2+|\Delta u|^2 dx$ is continuous and converges to $0$ as $|y|\to\infty$.
We first claim that $\rho(u(t),\delta)$ is bounded for all $\delta>0$.
Indeed, if this were not true, then there would exist $\delta>0$ and a sequence $\{t_n\}$ such that
\begin{align*}
      \int_{|x-y|<n}|u(t_n,x)|^2+|\Delta u(t_n,x)|^2 dx
      \leq (1-\delta)\int_{\R^d}|u(t_n,x)|^2+|\Delta u(t_n,x)|^2 dx
\end{align*}
for all $y\in\R^d$ and for all $n$. On the other hand, from Lemma \ref{lem:concentration_compactness} there exist $w_0$ and a sequence $\{Y_n\}$ such that
\begin{align*}
      u(t_n,\cdot+Y_n)\to w_0\ \ \ \mbox{in}\ H^2(\R^d).
\end{align*}
Hence, for any $y\in\R^d$ and $R>0$, we have
\begin{align*}
      \int_{|x-y|<R}|w_0|^2+|\Delta w_0|^2 dx
      =&\ \lim_{n\to\infty}\int_{|x-y|<R}|u(t_n,x+Y_n)|^2+|\Delta u(t_n,x+Y_n)|^2 dx \notag \\
      \leq&\ (1-\delta)\int_{\R^d}|u(t_n,x)|^2+|\Delta u(t_n,x)|^2 dx.
\end{align*}
Thus, $\|w_0\|_{H^2}^2\leq(1-\delta)\|w_0\|_{H^2}^2$, which implies that $\|w_0\|_{H^2}=0$.
This contradicts the fact that 
\begin{align*}
      S_{\omega}(w_0)=\lim_{n\to\infty}S_{\omega}(u(t_n))=L_{\omega}^*>0.
\end{align*}
Consequently, there exists $R(\delta)>0$ such that 
\begin{align*}
      \rho(u(t),\delta)<R(\delta)\ \ \ \mbox{for all}\ t\in\R.
\end{align*}
Similarly, we can show that there exists $m(\delta)>0$ such that 
\begin{align*}
      \lambda(u(t),R(\delta))>m(\delta)\ \ \ \mbox{for all}\ t\in\R.
\end{align*}
Fix $\delta<1/100$. Let $x(t)$ be such that 
\begin{align*}
      \lambda(u(t),R(\delta))
      =\int_{|x-x(t)|<R(\delta)}|u(t,x)|^2+|\Delta u(t,x)|^2 dx.
\end{align*}
We claim that $\{u(t,\cdot+x(t))\ |\ t\in\R\}$ is precompact in $H^2(\R^d)$.
Let $\{t_n\}$ be a sequence of times.
By Lemma \ref{lem:concentration_compactness}, there exists a sequence $\{Y_n\}$ and $w_0$ such that
\begin{align*}
      U(t_n):=u(t_n,\cdot+x(t_n)+Y_n)\to w_0\ \ \ \mbox{in}\ H^2(\R^d).
\end{align*}
In particular, $\{U(t_n)\}$ is a Cauchy sequence. Let $n_0$ be sufficiently large so that for all $n\leq n_0$,
\begin{align}\label{eq:upper_U}
      \|U(t_n)-U(t_{n_0})\|_{H^2}^2
      \leq \tfrac{m(\delta)}{10}
      <\tfrac{1}{20}(\|U(t_n)\|_{H^2}^2+\|U(t_{n_0})\|_{H^2}^2).
\end{align} 
Suppose that there exists a subsequence $\{Y_{n_k}\}$ such that $|Y_{n_k}-Y_{n_0}|\to\infty$ as $k\to\infty$.
Then for $|Y_{n_k}-Y_{n_0}|>2R(\delta)$, we have
\begin{eqnarray*}
      \lefteqn{ \|U(t_{n_k})-U(t_{n_0})\|_{H^2}^2 }\notag \\
      &=&\|U(t_{n_k})\|_{H^2}^2+\|U(t_{n_0})\|_{H^2}^2 \notag \\
      &&-2\re\int_{\R^d}\left(
            \Delta U(t_{n_k}) \Delta \overline{U(t_{n_0})}+U(t_{n_k})\overline{U(t_{n_0})}
      \right)dx \notag \\
      &=&\|U(t_{n_k})\|_{H^2}^2+\|U(t_{n_0})\|_{H^2}^2 \notag \\
      &&-2\re\left(\int_{|x-Y_{n_k}|<R(\delta)}+\int_{|x-Y_{n_k}|\geq R(\delta)}\right)\left(
            \Delta U(t_{n_k}) \Delta \overline{U(t_{n_0})}+U(t_{n_k})\overline{U(t_{n_0})}
      \right)dx \notag \\
      &\geq&\|U(t_{n_k})\|_{H^2}^2+\|U(t_{n_0})\|_{H^2}^2
      -2\delta^{\frac12}\|U(t_{n_k})\|_{H^2}\|U(t_{n_0})\|_{H^2},
\end{eqnarray*}
where we used 
\begin{eqnarray*}
      \lefteqn{ \int_{|x-Y_{n_k}|\geq R(\delta)}|\Delta U(t_{n_k})|^2+|U(t_{n_k})|^2 dx} \notag \\
      &=&\int_{|z|\geq R(\delta)}|\Delta u(t_{n_k},z+x(t_{n_k}))|^2
      +|u(t_{n_k},z+x(t_{n_k})|^2 dz \notag \\
      &\leq&\delta\|u(t_{n_k})\|_{H^2}^2
      =\delta\|U(t_{n_k})\|_{H^2}^2,
\end{eqnarray*}
which follows from $\rho(u(t_{n_k}),\delta)<R(\delta)$ and the definition of $\rho$.
Since $\delta<1/100$, we obtain
\begin{align*}
      \|U(t_{n_k})-U(t_{n_0})\|_{H^2}^2
      \geq \tfrac12(\|U(t_{n_k})\|_{H^2}^2+\|U(t_{n_0})\|_{H^2}^2)
\end{align*}
which contradicts \eqref{eq:upper_U}.
Consequently, $\{Y_n\}$ is a bounded sequence, and so we can assume that $Y_n\to Y_0$ passing to a subsequence if necessary.
Therefore, passing to a subsequence, we have
\begin{align*}
      u(t_n,\cdot+x(t_n))
      =U(t_n,\cdot-Y_n)
      \to w_0(\cdot-Y_0)\ \ \ \mbox{in}\ H^2(\R^d).
\end{align*}
Thus, we conclude that $\{u(t,\cdot+x(t))\ |\ t\in\R\}$ is precompact in $H^2(\R^d)$.
\qed

\section{Interaction virial identity}\label{sec:virial}

In this section, we establish the spatially localized interaction virial identity for the fourth-order Schr\"{o}dinger equation \eqref{eq:4NLS}.

\begin{lem}[Tightness in $H^2$]\label{lem:tightness}
Let $u$ be a global solution to \eqref{eq:4NLS} such that $\{u(t,\cdot+x(t))\ |\ t\in\R\}$ is precompact in $H^2(\R^d)$ for some function $x(\cdot)$. Then for all $\eta>0$, there exists $C(\eta)>0$ such that 
\begin{align*}
      \int_{|x-x(t)|\geq C(\eta)}|u(t,x)|^2+|\Delta u(t,x)|^2+|u(t,x)|^{p+1} dx
      \leq\eta
\end{align*}
for all $t\in\R$.
\end{lem}

\noindent
{\bf Proof.} Suppose that Lemma \ref{lem:tightness} does not hold. Then there exist $\eta_0>0$, $R_n>0$ and $t_n\in\R$ such that $R_n\rightarrow\infty$ and
\begin{align}\label{eq:>eta_0}
      \int_{|x-x(t_n)|\geq R_n}|u(t_n,x)|^2+|\Delta u(t_n,x)|^2+|u(t_n,x)|^{p+1} dx > \eta_0
\end{align}
for all $n$. Since $\{ u(t,\cdot+x(t_n))\ |\ t\in\R \}$ is precompact in $H^2(\R^d)$, there exists $\phi\in H^2(\R^d)$ such that $u(t_n,\cdot+x(t_n))$ has a subsequence that converges to $\phi$ in $H^2(\R^d)$. In particular, there exists $n_1$ such that
\begin{align}\label{eq:H^2_u-phi}
      \|u(t_n,\cdot+x(t_n))-\phi\|_{H^2}^2<\tfrac{1}{4}\eta_0
\end{align}
for all $n\geq n_1$. On the other hand, since $R_n\rightarrow\infty$, there exists $n_2$ such that
\begin{align}\label{eq:H^2(|x|>R_n)_phi}
      \int_{|x|\geq R_n} |\phi(x)|^2+|\Delta \phi(x)|^2+|\phi|^{p+1} dx<\tfrac{1}{4}\eta_0
\end{align}
for all $n\geq n_2$. By the triangle inequality and the Sobolev inequality, \eqref{eq:H^2_u-phi} and \eqref{eq:H^2(|x|>R_n)_phi} contradict \eqref{eq:>eta_0} when $n\geq\max\{n_1,n_2\}$.
\qed \\

\begin{prop}[Localized interaction virial identity]\label{prop:virial}
Let $d\geq1$, and $p$ satisfy \eqref{eq:intercritical}. Let $u$ be a global solution to \eqref{eq:4NLS} such that $\|u\|_{L_t^{\infty}H_x^2(\R\times\R^d)}<\infty$ and $\{u(t,\cdot+x(t))\ |\ t\in\R\}$ is precompact in $H^2(\R^d)$ for some function $x(\cdot)$. Then for any $\epsilon>0$, 
\begin{align}\label{eq:interact_virial}
      \left|
      M(u)\tfrac{1}{|I|}\int_I K(u(t)) dt
      +2P(u)\cdot \tfrac{1}{|I|}\int_I \im\int_{\R^d}\Delta \overline{u}(t,x)\nabla u(t,x)dx dt
      \right|
      \leq \epsilon
\end{align}
holds for all sufficiently large intervals $I\subset\R$.
\end{prop}

\noindent
{\bf Proof.} Let $\psi\in C_0^{\infty}([0,\infty))$ be a smooth, non-increasing function and such that $\psi(s)=1$ for $0\leq s\leq1$ and $\psi(s)=0$ for $s\geq2$. For each $R>0$, we define
\begin{align*}
      \Psi_R(x):=R^2\int_0^{\frac{|x|}{R}}\int_0^r \psi(s) dsdr.
\end{align*}
Then for all $j\geq1$, we have
\begin{align}
      \|\nabla^j\Psi_R\|_{L^{\infty}}\lesssim R^{2-j}. \label{eq:Psi_R_0}
\end{align}
Let we denote radial derivative by $\pa_r=\tfrac{\pa}{\pa r}$ with $r=|x|$. Since $\psi(s)=1$ for $0\leq s\leq 1$, we see that
\begin{align}
      \tfrac{\pa_r\Psi_R}{|x|}-1
      =&\ \tfrac{R}{|x|}\int_0^{\frac{|x|}{R}}\psi(s)ds-1 =0, \label{eq:Psi_R_1} \\
      \pa_r^2\Psi_R-\tfrac{\pa_r\Psi_R}{|x|}
      =&\ \left\{\psi\left(\tfrac{|x|}{R}\right)-1\right\}+\left\{1-\tfrac{\pa_r\Psi_R}{|x|}\right\} =0
      \label{eq:Psi_R_2}
\end{align}
for $|x|\leq R$. Moreover, since $\psi$ is non-increasing, we see that
\begin{align}
      \pa_r^2\Psi_R-\tfrac{\pa_r\Psi_R}{|x|}
      = \tfrac{R}{|x|}\int_0^{\frac{|x|}{R}}\left\{\psi\left(\tfrac{|x|}{R}\right)-\psi(s)\right\}ds \leq0
\end{align}
for all $x\in\R^d$. For $u\in C(\R,H^2(\R^d))$ and $R>0$, we define
\begin{align*}
      V_R^{interact}(t)
      :=&\ 2\iint_{\R^d\times\R^d}|u(t,y)|^2\nabla\Psi_R(x-y)
      \cdot\im[\overline{u}(t,x)\nabla u(t,x)]dxdy \\
      =&\ \int_{\R^d}|u(t,y)|^2\langle iu(t)|Z_y u(t) \rangle dy,
\end{align*}
where $\langle f|g \rangle:=\re\int_{\R^d}f(x)\overline{g}(x) dx$ and
\begin{align*}
      Z_y:= \nabla\Psi_R(x-y)\cdot\nabla_x+\nabla_x\cdot\nabla\Psi_R(x-y).
\end{align*}
We note that by using $Z_y=2\nabla\Psi_R(x-y)\cdot\nabla_x+\Delta\Psi_R(x-y)$, we have
\begin{align*}
      \langle iu(t)|Z_y u(t) \rangle
      = 2\im\int_{\R^d}\nabla\Psi_R(x-y)\cdot\overline{u}(t,x)\nabla u(t,x)dx.
\end{align*}
If $u$ solves \eqref{eq:4NLS}, then 
\begin{align}
      \tfrac{d}{dt}V_R^{interact}(t)
      =&\ 2\im\int_{\R^d}\overline{u}(t,y)\Delta^2 u(t,y)\langle iu(t)|Z_y u(t) \rangle dy \label{eq:V1} \\
      &\ +\int_{\R^d}|u(t,y)|^2\pa_t\langle iu(t)|Z_y u(t) \rangle dy. \label{eq:V2}
\end{align}
First, we estimate \eqref{eq:V1}. Using the integration by parts, we obtain
\begin{align}
      \eqref{eq:V1}
      =&\ 4\im\int_{\R^d}\Delta u(t,y)\nabla\overline{u}(t,y)
      \cdot\nabla_y\langle iu(t)|Z_y u(t) \rangle dy \label{eq:V11} \\
      &\ -2\im\int_{\R^d}\overline{u}(t,y)\nabla u(t,y)
      \cdot\nabla_y\Delta_y \langle iu(t)|Z_y u(t) \rangle dy. \label{eq:V12}
\end{align}
Applying \eqref{eq:Psi_R_1}, \eqref{eq:Psi_R_2} and
\begin{align}\label{eq:Psi_R_kl}
      \pa_{kl}\Psi_R
      =\left(\delta_{kl}-\tfrac{x_k x_l}{|x|^2}\right)\tfrac{\pa_r\Psi_R}{|x|}
      +\tfrac{x_k x_l}{|x|^2}\pa_r^2\Psi_R,
\end{align}
we have
\begin{align*}
      \eqref{eq:V11}
      =&\ 8\iint_{\R^d\times\R^d}\im[\Delta\overline{u}\pa_k u](t,y)
      \pa_{kl}\Psi_R(x-y)\im[\overline{u}\pa_l u](t,x) dxdy \\
      =&\ 8P(u)\cdot\im\int_{\R^d}\Delta \overline{u}(t,x)\nabla u(t,x) dx + A_R^1(t),
\end{align*}
where $z=x-y$ and 
\begin{align*}
      A_R^1(t) 
      =&\ 8\iint_{|z|\geq R}\left\{ \im[\Delta\overline{u}\nabla u](t,y)
      \left(\tfrac{\pa_r\Psi_R(z)}{|z|}-1\right)
      \im[\overline{u}\nabla u](t,x) \right. \notag \\
      &\ \left. +\im[\Delta\overline{u}\pa_k u](t,y)
      \tfrac{z_k z_l}{|z|^2}\left(\pa_r^2\Psi_R(z)-\tfrac{\pa_r\Psi_R(z)}{|z|}\right)
      \im[\overline{u}\pa_l u](t,x) \right\} dxdy.
\end{align*}
By using Lemma \ref{lem:tightness}, we have 
\begin{align*}
      A_R^1(t)
      \lesssim&\ \|u\|_{L_t^{\infty}H_x^2(\R\times \R^d)}^{\frac32}
      \left(\int_{|y-x(t)|\geq \frac{R}{2}}|\Delta u(t,y)|^2 dy\right)^{\frac12} \\
      &\ +\|u\|_{L_t^{\infty}H_x^2(\R\times \R^d)}^{\frac32}
      \left(\int_{|x-x(t)|\geq \frac{R}{2}}|u(t,x)|^2 dx\right)^{\frac12} \\
      =&\ o_R(1),
\end{align*}
where $o_R(1)\to0$ as $R\to\infty$ uniformly in $t\in\R$.
Using \eqref{eq:Psi_R_0}, we have
\begin{align*}
      \eqref{eq:V12}
      \lesssim R^{-2}M(u)\|u\|_{L_t^{\infty}\dot{H}_x^1(\R\times \R^d)}^2
      = o_R(1).
\end{align*}
Therefore, we obtain
\begin{align}\label{eq:est_V1}
      \eqref{eq:V1}
      = 8P(u)\cdot\im\int_{\R^d}\Delta\overline{u}(t,x)\nabla u(t,x)dx+o_R(1).
\end{align}
Next, we estimate \eqref{eq:V2}. Since $Z_y^*=-Z_y$, we have
\begin{align*}
      \pa_t\langle iu|Z_y u \rangle
      =&\ 2\langle iu_t|Z_y u \rangle
      =2\langle \Delta^2 u-|u|^{p-1}u|Z_y u \rangle \\
      =&\ \langle [\Delta^2,Z_y]u|u \rangle-\langle [|u|^{p-1},Z_y]u|u \rangle,
\end{align*}
where $[A,B]=AB-BA$. To estimate $\langle [\Delta^2,Z_y]u|u \rangle$, we use two commutator identities. First, we have
\begin{align}
      [\Delta,Z_y]=4\pa_k(\pa_{kl}\Psi_R)\pa_l+(\Delta^2\Psi_R). \label{eq:commutator_1}
\end{align}
Second, by using \eqref{eq:commutator_1} and the identities $\Delta A+A \Delta=2\pa_k A\pa_k+[\pa_k,[\pa_k,A]]$ for an operator $A$, we have
\begin{align}
      [\Delta^2,Z_y]
      =&\ \Delta[\Delta,Z_y]+[\Delta,Z_y]\Delta \notag \\
      =&\ 2\pa_k[\Delta,Z_y]\pa_k+[\pa_k,[\pa_k,[\Delta,Z_y]]] \notag \\
      =&\ 8\pa_{kl}(\pa_{lm}\Psi_R)\pa_{mk}+4\pa_k(\pa_{kl}\Delta\Psi_R)\pa_l
      +2\pa_k(\Delta^2\Psi_R)\pa_k+(\Delta^3\Psi_R). \label{eq:commutator_2}
\end{align}
By using \eqref{eq:commutator_2}, we integrate by parts to obtain
\begin{align}
      \langle [\Delta^2,Z_y]u|u \rangle
      =&\ 8\re\int_{\R^d}\pa_{kl}\overline{u}(\pa_{lm}\Psi_R)\pa_{mk}udx \label{eq:V21} \\
      &\ -4\re\int_{\R^d}\pa_k\overline{u}(\pa_{kl}\Delta\Psi_R)\pa_l udx 
      -2\re\int_{\R^d}\Delta^2\Psi_R|\nabla u|^2dx \label{eq:V22} \\
      &\ +\int_{\R^d}\Delta^3\Psi_R|u|^2dx. \label{eq:V23}
\end{align}
Applying \eqref{eq:Psi_R_kl}, we have
\begin{align*}
      \eqref{eq:V21}
      =&\ 8\int_{\R^d}|\Delta u(t,x)|^2 dx 
      + 8\int_{\R^d}\left\{ 
      \left(\tfrac{\pa_r\Psi_R(z)}{|z|}-1\right)|\Delta u(t,x)|^2 \right. \notag \\
      &\ \left. +\left(\pa_r^2\Psi_R(z)-\tfrac{\pa_r\Psi_R(z)}{|z|}\right)
      |\tfrac{z}{|z|}\cdot\nabla\pa_k u(t,x)|^2 
      \right\} dx.
\end{align*}
with $z=x-y$. By using \eqref{eq:Psi_R_0}, we have
\begin{align*}
      \eqref{eq:V22}+\eqref{eq:V23}
      \lesssim&\ R^{-2}\|u\|_{L_t^{\infty}\dot{H}_x^1(\R\times\R^d)}^2
      +R^{-4}M(u).
\end{align*}
By using $[|u|^{p-1},Z_y]=-2\nabla\Psi_R\cdot\nabla(|u|^{p-1})$, we integrate by parts to obtain
\begin{align}
      -\langle [|u|^{p-1},Z_y]u|u \rangle
      =&\ 2\int_{\R^d}|u|^2\nabla\Psi_R\cdot\nabla(|u|^{p-1})dx \notag \\
      =&\ \tfrac{2(p-1)}{p+1}\int_{\R^d}\nabla\Psi_R\cdot\nabla(|u|^{p+1})dx \notag \\
      =&\ -\tfrac{2(p-1)}{p+1}\int_{\R^d}\Delta\Psi_R|u|^{p+1}dx. \label{eq:V24}
\end{align}
Applying $\Delta\Psi_R=(d-1)\tfrac{\pa_r\Psi_R}{|x|}+\pa_r^2\Psi_R$, we have
\begin{align*}
      \eqref{eq:V24}
      =&\ -\tfrac{2d(p-1)}{p+1}\int_{\R^d}|u(t,x)|^{p+1}dx
      -\tfrac{2(p-1)}{p+1}\int_{\R^d}\left\{ d\left(\tfrac{\pa_r\Psi_R(z)}{|z|}-1\right) \right. \notag \\
      &\ \left. +\left(\pa_r^2\Psi_R(z)-\tfrac{\pa_r\Psi_R(z)}{|z|}\right) \right\}|u(t,x)|^{p+1}dx.
\end{align*}
Hence, by applying \eqref{eq:Psi_R_1} and \eqref{eq:Psi_R_2}, we have
\begin{align*}
      \eqref{eq:V2}
      =&\ 4M(u)K(u(t))+A_R^2(t)
      +O\left(R^{-2}M(u)\|u\|_{L_t^{\infty}\dot{H}_x^1(\R\times\R^d)}^2+R^{-4}M(u)^2\right),
\end{align*}
where $z=x-y$ and
\begin{align*}
      A_R^2(t)
      =&\ 4\iint_{|z|\geq R}|u(t,y)|^2
      \left\{
      \left(\tfrac{\pa_r\Psi_R(z)}{|z|}-1\right)
      \left(2|\Delta u(t,x)|^2-\tfrac{d(p-1)}{2(p+1)}|u(t,x)|^{p+1}\right)
      \right. \notag \\
      &\ \left. 
      +\left(\pa_r^2\Psi_R(z)-\tfrac{\pa_r\Psi_R(z)}{|z|}\right)
      \left(2|\tfrac{z}{|z|}\cdot\nabla\pa_k u(t,x)|^2-\tfrac{p-1}{2(p+1)}|u(t,x)|^{p+1}\right)
      \right\} dxdy.
\end{align*}
By using Lemma \ref{lem:tightness} and the Sobolev inequality, we have
\begin{align*}
      A_R^2(t)
      \lesssim&\ \|u\|_{L_t^{\infty}H_x^2(\R\times \R^d)}^2
      (1+\|u\|_{L_t^{\infty}H_x^2(\R\times \R^d)}^{p-1})
      \int_{|y-x(t)|\geq \frac{R}{2}} |u(t,y)|^2 dy \\
      &\ +M(u)\int_{|x-x(t)|\geq \frac{R}{2}}
      \left( |\Delta u(t,x)|^2+|u(t,x)|^{p+1} \right) dx \\
      =&\ o_R(1),
\end{align*}
where $o_R(1)\to0$ as $R\to\infty$ uniformly in $t\in\R$. Therefore, we obtain
\begin{align}\label{eq:est_V2}
      \eqref{eq:V2}
      =4M(u)K(u(t))+o_R(1).
\end{align}
From \eqref{eq:V1}, \eqref{eq:V2}, \eqref{eq:est_V1} and \eqref{eq:est_V2}, we have
\begin{align}\label{eq:est_virial}
      \tfrac{d}{dt}V_R^{interact}(t)
      = 4M(u)K(u(t))+8P(u)\cdot\im\int_{\R^d}\Delta\overline{u}(t,x)\nabla u(t,x)dx+o_R(1),
\end{align}
where $o_R(1)\to0$ as $R\to\infty$ uniformly in $t\in\R$. Integrating \eqref{eq:est_virial} on interval $I$, we obtain
\begin{align*}
      &\ \left|M(u)\int_{I}K(u(t))dt
      +2P(u)\cdot\int_{I}\im\int_{\R^d}\Delta\overline{u}(t,x)\nabla u(t,x)dxdt\right| \notag \\
      \lesssim&\ \sup_{t\in I}|V_R^{interact}(t)|+o_R(1)|I|.
\end{align*}
By the definition of $V_R^{interact}$, we have
\begin{align*}
      \sup_{t\in \R}|V_R^{interact}(t)|\lesssim R\|u\|_{L_t^{\infty}H_x^2(\R\times\R^d)}^4.
\end{align*}
Thus, for any $\epsilon>0$, we can choose $R>0$ sufficiently large so that there exists $C_{\epsilon}>0$ such that 
\begin{align*}
      &\ \left|M(u)\int_{I}K(u(t))dt
      +2P(u)\cdot\int_{I}\im\int_{\R^d}\Delta\overline{u}(t,x)\nabla u(t,x)dxdt\right|
      \leq C_{\epsilon}+\tfrac{\epsilon}{2}|I|
\end{align*}
holds for any interval $I\subset\R$. Therefore, for any $\epsilon>0$, \eqref{eq:interact_virial} holds for all sufficiently large intervals $I\subset\R$.
\qed

\begin{cor}\label{cor:zero_mom}
Let $d\geq1$, and $p$ satisfy \eqref{eq:intercritical}. 
Let $u$ be a global solution to \eqref{eq:4NLS} with initial data $u_0$ satisfying \eqref{eq:below_gs}  such that $\{u(t,\cdot+x(t))\ |\ t\in\R\}$ is precompact in $H^2(\R^d)$ for some function $x(\cdot)$.
If $P(u_0)=0$ or 
\begin{align}\label{eq:C=0}
      \int_0^T\im\int_{\R^d}\Delta\overline{u}(t,x)\nabla u(t,x)dxdt
      =o(T)
\end{align}
as $T\to\infty$, then $M(u)=0$.
\end{cor}

\noindent
{\bf Proof.}
If $K(u_0)=0$, then from $S_{\omega}(u_0)<S_{\omega}(Q_{\omega})$ and \eqref{eq:def_gs_2}, we see that $u_0=0$ and so $M(u)=0$.
If $K(u_0)>0$, we suppose that $M(u)\neq0$.
If $P(u_0)=0$ or \eqref{eq:C=0}, by using Proposition \ref{prop:virial}, we have
\begin{align}\label{eq:K=0}
      \tfrac{1}{T}\int_0^T K(u(t))dt= o(1)
\end{align}
as $T\to\infty$. Since $K(u_0)>0$, then we see that
\begin{align*}
      E(u)= E(u_0)
      = \tfrac{s_c}{d}\|\Delta u_0\|_{L^2}^2+\tfrac{2}{d(p-1)}K(u_0)
      > 0.
\end{align*} 
Moreover, by Lemma \ref{lem:invariant_set}, $K(u(t))>0$ for all $t\in\R$ and so by Lemma \ref{lem:coercivity}, 
\begin{align*}
      K(u(t))
      >&\  2\left\{1-\left(\tfrac{S_{\omega}(u)}{S_{\omega}(Q_{\omega})}\right)^{\frac{p-1}{2}}\right\}
      \|\Delta u(t)\|_{L^2}^2 
      > 4\left\{1-\left(\tfrac{S_{\omega}(u)}{S_{\omega}(Q_{\omega})}\right)^{\frac{p-1}{2}}\right\}
      E(u)>0
\end{align*}
for all $t\in\R$. This contradicts to \eqref{eq:K=0} for sufficiently large $T>0$.
\qed

\section{Proof of Theorem \ref{thm:main}}\label{sec:proof}

In this section, we completes the proof of Theorem \ref{thm:main}. The proof is based on the argument in \cite{PS}. \\

\noindent
{\bf Proof of Theorem \ref{thm:main}.}
If $L_{\omega}^*\geq S_{\omega}(Q_{\omega})$, Theorem \ref{thm:main} is proven.
To complete the proof, suppose that $L_{\omega}^*<S_{\omega}(Q_{\omega})$ and let $u$ be a global non-scattering solution to \eqref{eq:4NLS} in Proposition \ref{prop:critical_sol} with initial data $u_0$ satisfying $S_{\omega}(u_0)=L_{\omega}^*$ and $K(u_0)\geq0$. Then, there exists $x:\R\to\R^d$ such that $\{u(t,\cdot+x(t))\ |\ t\in\R\}$ is precompact in $H^2(\R^d)$. If $K(u_0)=0$, then from $S_{\omega}(u_0)<S_{\omega}(Q_{\omega})$ and \eqref{eq:def_gs_2}, we see that $u_0=0$ and so $u(t)=0$ for all $t\in\R$, which contradicts to \eqref{eq:X_of_u}. Hence, we may assume $K(u_0)>0$ in the following.

If $P(u)=0$, Corollary \ref{cor:zero_mom} implies that $M(u)=0$, which contradicts to \eqref{eq:X_of_u}. Hence, without loss of generality, we can assume that the momentum vector $P(u)$ is nonzero and parallel to the first vector $e_1$, namely, $P(u)=\bm{p}(u)e_1$ with $\bm{p}(u)>0$.

For a real-valued function $\phi:\R^d\to\R$, we have
\begin{align*}
      |\Delta(e^{i\phi(x)}u)|^2
      =&\ |\Delta u|^2
      -2|\nabla\phi|^2\re[\overline{u}\Delta u]
      +(|\nabla\phi|^4+|\Delta\phi|^2)|u|^2 \notag \\
      &\ +4|\nabla\phi\cdot\nabla u|^2 
      +2\Delta\phi\nabla\phi\cdot\nabla|u|^2
      +4\im[\nabla\phi\nabla\overline{u}\Delta u] \notag \\
      &\ +2\Delta\phi\im[\overline{u}\Delta u]
      -4|\nabla\phi|^2\nabla\phi\cdot\im[u\nabla\overline{u}].
\end{align*}
Integrating this and applying the integration by parts, we have 
\begin{align*}
      \int_{\R^d}|\Delta(e^{i\phi(x)}u)|^2dx
      =&\ \int_{\R^d}|\Delta u|^2dx
      +2\int_{\R^d}|\nabla\phi|^2|\nabla u|^2 dx 
      +4\int_{\R^d}|\nabla\phi\cdot\nabla u|^2dx \\
      &\ +\int_{\R^d}(|\nabla\phi|^4+|\Delta\phi|^2-2\nabla\Delta\phi\cdot\nabla\phi
      -2\Delta|\nabla\phi|^2)|u|^2dx \\
      &\ +4\im\int_{\R^d}\nabla\phi\cdot\Delta u\nabla\overline{u}dx  
      -2\im\int_{\R^d}\nabla\Delta\phi\cdot\overline{u}\nabla udx \\
      &\ -4\im\int_{\R^d}|\nabla\phi|^2\nabla\phi\cdot u\nabla\overline{u}dx.
\end{align*}
Let we take $\phi(x)=Xx_1$ for $X\in\R$ and denote $\tilde{u}=e^{iXx_1}u$. Then,
\begin{align}
      K(\tilde{u})=&\ K(e^{iXx_1}u) \notag \\
      =&\ 2\|\Delta(e^{iXx_1}u)\|_{L^2}^2-\tfrac{d(p-1)}{2(p+1)}\|u\|_{L^{p+1}}^{p+1} \notag \\
      =&\ K(u)+8C(u)X+12G(u)X^2+8\bm{p}(u)X^3+2M(u)X^4, \label{eq:K_tilde_u}
\end{align}
where
\begin{align*}
      C(u):=&\ \im\int_{\R^d}\Delta u\pa_1\overline{u}dx, \\
      G(u):=&\ \tfrac13\|\nabla u\|_{L^2}^2+\tfrac23\|\pa_1 u\|_{L^2}^2.
\end{align*}

Let $I_k$ be a sequence of intervals such that \eqref{eq:interact_virial} in Proposition \ref{prop:virial} holds with $\epsilon=1/k$ and take the average of \eqref{eq:K_tilde_u} on $I_k$. Then we obtain the sequence of polynomials
\begin{align*}
      F_k(X)
      :=&\ \tfrac{1}{|I_k|}\int_{I_k}K(u(t))dt
      +8\left(\tfrac{1}{|I_k|}\int_{I_k}C(u(t))dt\right)X 
      \notag \\
      &\ +12\left(\tfrac{1}{|I_k|}\int_{I_k}G(u(t))dt\right)X^2 
      +8\bm{p}(u)X^3
      +2M(u)X^4.
\end{align*}
Since all the coefficients of $F_k$ are uniformly bounded, without loss of generality, we can assume that $F_k$ converge pointwise to a polynomial
\begin{align*}
      F(X)= \langle K(u) \rangle + 8\langle C(u) \rangle X + 12\langle G(u) \rangle X^2
      + 8\bm{p}(u)X^3 + 2M(u)X^4,
\end{align*}
where
\begin{align*}
      \langle K(u) \rangle
      :=&\ \lim_{k\to\infty}\tfrac{1}{|I_k|}\int_{I_k}K(u(t))dt, \\
      \langle C(u) \rangle
      :=&\ \lim_{k\to\infty}\tfrac{1}{|I_k|}\int_{I_k}C(u(t))dt, \\
      \langle G(u) \rangle
      :=&\ \lim_{k\to\infty}\tfrac{1}{|I_k|}\int_{I_k}G(u(t))dt.
\end{align*}
By using Proposition \ref{prop:virial} with $\epsilon=1/k$ and $I=I_k$, we have 
\begin{align*}
      M(u)\langle K(u) \rangle = 2\bm{p}(u)\langle C(u) \rangle.
\end{align*}
By using this, we obtain
\begin{align}
      \tfrac{1}{\langle K(u) \rangle}F\left(-\left(\tfrac{\langle K(u) \rangle}{2M(u)}\right)^{\frac14}\right)
      =&\ 6\left(\tfrac{\langle G(u) \rangle}{\sqrt{\frac12\langle K(u) \rangle M(u)}}-1\right)
      -4\left(\sqrt{\alpha}-\tfrac{1}{\sqrt{\alpha}}\right)^2 \notag \\
      \leq&\ 6\left(\tfrac{\langle G(u) \rangle}{\sqrt{\frac12\langle K(u) \rangle M(u)}}-1\right),
      \label{eq:estimate_F}
\end{align}
where $\alpha=\bm{p}(u)/(\frac12\langle K(u) \rangle M(u)^3)^{\frac14}$.

\begin{lem}\label{lem:kappa}
If $S_{\omega}(u)<S_{\omega}(Q_{\omega})$ and $K(u)>0$, then
\begin{align}\label{eq:kappa}
      K(\tilde{u})> \tfrac{1-\kappa}{\kappa}(2\|\Delta u\|_{L^2}^2-K(u))
\end{align}
with
\begin{align*}
      \kappa:= \left(\tfrac{S_{\omega}(u)}{S_{\omega}(Q_{\omega})}\right)^{\frac{p-1}{2}}< 1.
\end{align*}
\end{lem}

\noindent
{\bf Proof.}
By using Lemma \ref{lem:coercivity}, we have
\begin{align*}
      K(u)> 2(1-\kappa)\|\Delta u\|_{L^2}^2.
\end{align*}
If $\|\Delta\tilde{u}\|_{L^2}^2\geq\|\Delta u\|_{L^2}^2$, then $K(\tilde{u})\geq K(u)$ and hence
\begin{align*}
      K(\tilde{u})\geq K(u)> \tfrac{1-\kappa}{\kappa}(2\|\Delta u\|_{L^2}^2-K(u)).
\end{align*}
Thus, \eqref{eq:kappa} holds. On the other hand, if $\|\Delta\tilde{u}\|_{L^2}^2<\|\Delta u\|_{L^2}^2$, we consider the two cases, $K(\tilde{u})>0$ and $K(\tilde{u})\leq0$. In the former case, by using Lemma \ref{lem:coercivity}, we have
\begin{align*}
      K(\tilde{u})
      >&\ 2\left\{
      1-\left(\tfrac{S_{\omega}(\tilde{u})}{S_{\omega}(Q_{\omega})}\right)^{\frac{p-1}{2}}
      \right\}\|\Delta\tilde{u}\|_{L^2}^2 \notag \\
      >&\ 2(1-\kappa)\|\Delta\tilde{u}\|_{L^2}^2 \notag \\
      =&\ (1-\kappa)(K(\tilde{u})-K(u)+2\|\Delta u\|_{L^2}^2),
\end{align*}
which implies \eqref{eq:kappa}. In the latter case, for any $0<\epsilon<K(u)$, we can take $X^{\prime}\in\R$ such that $K(e^{iX^{\prime}x_1}u)=\epsilon$ and $\|\Delta(e^{iX^{\prime}x_1}u)\|_{L^2}^2<\|\Delta u\|_{L^2}^2$. By using Lemma \ref{lem:coercivity}, we have
\begin{align*}
      K(e^{iX^{\prime}x_1}u)
      >&\ 2\left\{
      1-\left(\tfrac{S_{\omega}(e^{iX^{\prime}x_1}u)}{S_{\omega}(Q_{\omega})}\right)^{\frac{p-1}{2}}
      \right\}\|\Delta(e^{iX^{\prime}x_1}u)\|_{L^2}^2 \notag \\
      >&\ 2(1-\kappa)\|\Delta(e^{iX^{\prime}x_1}u)\|_{L^2}^2 \notag \\
      =&\ (1-\kappa)(K(e^{iX^{\prime}x_1}u)-K(u)+2\|\Delta u\|_{L^2}^2).
\end{align*}
Hence,
\begin{align*}
      \epsilon= K(e^{iX^{\prime}x_1}u)
      > \tfrac{1-\kappa}{\kappa}(2\|\Delta u\|_{L^2}^2-K(u)).
\end{align*}
Letting $\epsilon\to0$, we have $\kappa\geq1$, which is a contradiction.
\qed \\

Taking the large time average of \eqref{eq:kappa} in Lemma \ref{lem:kappa}, we have
\begin{align*}
      F(X)\geq \tfrac{1-\kappa}{\kappa}(\langle D(u) \rangle-\langle K(u) \rangle),
\end{align*}
where 
\begin{align*}
      \langle D(u) \rangle
      := \lim_{k\to\infty}\tfrac{1}{|I_k|}\int_{I_k}D(u(t))dt, \ \ \ 
      D(u)
      := 2\|\Delta u\|_{L^2}^2.
\end{align*}
Hence, by taking $X=-\left(\frac{\langle K(u) \rangle}{2M(u)}\right)^{\frac14}$ and using \eqref{eq:estimate_F}, we obtain
\begin{align}\label{eq:LTA_kappa}
      6\left(\tfrac{\langle G(u) \rangle}{\sqrt{\frac12\langle K(u) \rangle M(u)}}-1\right)
      \geq \tfrac{1-\kappa}{\kappa}\left(\tfrac{\langle D(u) \rangle}{\langle K(u) \rangle}-1\right).
\end{align}
Let $\Lambda>1$ be such that $\langle D(u) \rangle=\Lambda^2\langle K(u) \rangle$. Applying this to \eqref{eq:LTA_kappa} and using the estimate $\langle G(u) \rangle \leq \sqrt{\frac12\langle D(u) \rangle M(u)}$, we get
\begin{align*}
      6(\Lambda-1)\geq \tfrac{1-\kappa}{\kappa}(\Lambda^2-1),
\end{align*}
which implies that
\begin{align*}
      \tfrac{1-\kappa}{\kappa}\leq \tfrac{6}{\Lambda+1}< 3.
\end{align*}
Thus, we have $4\kappa>1$. This proves that
\begin{align*}
      L_{\omega}^*=S_{\omega}(u)>\left(\tfrac14\right)^{\frac{2}{p-1}}S_{\omega}(Q_{\omega}),
\end{align*}
which completes the proof of Theorem \ref{thm:main}.
\qed

\appendix 
\renewcommand{\theequation}{A.\arabic{equation}} 
\setcounter{equation}{0} 

\subsection*{Acknowledgement}
The author is grateful to Professor Satoshi Masaki for many helpful suggestions.
He also expresses his deep thanks to Professor Kenji Nakanishi for pointing out an error in the author's previous study, which is the beginning of the present study.
The author was supported by JSPS KAKENHI Grant Number 23K13003.


\subsection*{Conflict of interest}
The author was supported by JSPS KAKENHI Grant Number 23K13003.

\subsection*{Data availability}
Data sharing is not applicable to this article as no new data were created or analyzed in this
study.


\end{document}